\documentclass{amsart}
\usepackage[utf8]{inputenc}
\usepackage[T1]{fontenc}
\usepackage[english]{babel}
\usepackage{amsmath}
\usepackage{amsfonts}
\usepackage{amssymb}
\usepackage{graphicx}
\usepackage[colorlinks=true, allcolors=blue]{hyperref}

\usepackage{xcolor}

\usepackage{amsthm}
\newtheorem{theorem}{Theorem}
\newtheorem{lemma}{Lemma}
\newtheorem{proposition}{Proposition}

\theoremstyle{definition}
\newtheorem{definition}{Definition}

\theoremstyle{remark}
\newtheorem{remark}{Remark}

\DeclareMathOperator{\Id}{Id}
\DeclareMathOperator{\Span}{Span}

\title[Identities of an algebra of upper triangular matrices]{Gradings, graded  identities, $*$-identities   and graded $*$-identities of an algebra of upper triangular matrices}
\author[J. A. Gomez Parada]{Jonatan Andres Gomez Parada}
\address{Department of Mathematics \\ IMECC, UNICAMP \\ 
Sérgio Buarque de Holanda, 651, 13083-859 Campinas, SP, Brazil}
\email{j211980@dac.unicamp.br}
\thanks{This study was financed in part by the Coordena\c c\~ao de Aperfei\c coamento de Pessoal de N\'{\i}vel Superior - Brasil (CAPES) -
Finance Code 001.}
\author[P. Koshlukov]{Plamen Koshlukov}
\address{Department of Mathematics \\ IMECC, UNICAMP \\ 
Sérgio Buarque de Holanda, 651, 13083-859 Campinas, SP, Brazil}
\email{plamen@unicamp.br}
\thanks{P. Koshlukov was partially supported by FAPESP Grant 2018/23690-6 and by CNPq Grant 307184/2023-4}

\date{}
\begin{document}

\begin{abstract}
Let $K \langle X\rangle$  be the free associative algebra freely generated over the field $K$ by the countable set $X = \{x_1, x_2, \ldots\}$. If $A$ is an associative $K$-algebra, we say that a polynomial $f(x_1,\ldots, x_n) \in  K \langle X\rangle$ is a polynomial identity, or simply an identity in $A$ if $f(a_1,\ldots, a_n) = 0$ for every $a_1$, \dots, $a_n \in  A$.

Consider $\mathcal{A}$  the subalgebra of $UT_3(K)$ given by:
\[ \mathcal{A} =   K(e_{1,1} + e_{3,3}) \oplus Ke_{2,2} \oplus Ke_{2,3} \oplus Ke_{3,2} \oplus Ke_{1,3} , \]
where $e_{i,j}$ denote the matrix units. We investigate the gradings on the algebra $\mathcal{A}$, determined by an abelian group, and prove that these gradings are elementary. Furthermore, we compute a basis for the $\mathbb{Z}_2$-graded identities of $\mathcal{A}$, and also for the $\mathbb{Z}_2$-graded identities with graded involution. Moreover, we describe the cocharacters of this algebra.
\end{abstract}



\maketitle

\section*{Introduction}

Let  $K$ be a field and let $K\langle X \rangle$ be the free associative algebra freely generated by the countable set of indeterminates $X$. One can view $K\langle X\rangle$ as the set of all polynomials in the non-commuting variables from the set $X$. Given an associative algebra  $A$, a polynomial $f$ in $K\langle X\rangle$ is called a {\sl polynomial identity} of $A$ if $f$ evaluates to zero when its variables are substituted with arbitrary elements of $A$. An algebra satisfying a non-zero polynomial identity is called a {\sl PI-algebra}. The set of all identities for $A$ is denoted by $T(A)$. Clearly $T(A)$ is an ideal in $K\langle X\rangle$. Moreover it is closed under endomorphisms of the free algebra $K\langle X\rangle$. It can be seen that every such ideal coincides with $T(A)$ for some $A$.
Among the algebras that satisfy non-zero polynomial identities, those that have been of greatest interest in the development of the theory of polynomial identities and that will play a significant role in this paper include the Grassmann algebra, the full matrix algebras $M_n(K)$, and $UT_n(K)$, the algebra of the upper triangular matrices of order $n$. 

One initial problem to be considered in the theory of algebras with polynomial identities, which we will address here, is determining the set of all identities satisfied by a particular algebra, as well as a generating set for them. In the case of a PI-algebra $A$ over a field of characteristic zero, it is well known that the polynomial identities of $A$ follow from the multilinear ones. Therefore, we can restrict our study to multilinear polynomials.

In the particular case of the algebra $UT_n(K)$, its polynomial identities are well known. If $K$ is any infinite field, then a basis for the polynomials identities of $UT_n(K)$ is given by the polynomial $[x_1, x_2] \cdots [x_{2k-1}, x_{2k}]$. Considering the graded case, A. Valenti in \cite{valenti2002graded} studied the graded identities and the graded cocharacters of $UT_2(K)$. In \cite{di2004} the authors classified the elementary gradings on the upper triangular matrix algebras, and described their graded identities. In \cite{valenti2007group}  Valenti and  Zaicev proved that every group grading on the algebra $UT_n(K)$  is isomorphic to an elementary grading. Later on Cirrito  considered in \cite{cirrito2013gradings} the graded identities of $UT_3(K)$.

Involutions on the algebra $UT_n(K)$ have also been extensively studied. An {\sl involution} on an algebra $A$ is an antiautomorphism of order two, that is, a linear map $* \colon A \to  A$ satisfying $(ab)^{*} = b^{*}a^{*}$ and $(a^{*})^{*} = a$, for every $a$, $b \in  A$. In the classification of involutions, such ones are called involutions of the first kind. The involutions on $UT_n(K)$ were described in \cite{di2006involutions}. In that paper, the authors also described the identities with involution for the algebras $UT_2(K)$ and $UT_3(K)$. Graded involutions on $UT_n(K)$ were described by Valenti and Zaicev in \cite{valenti2009graded}.

We are interested in a particular subalgebra of $UT_3(K)$. Consider $\mathcal{A}$ the subalgebra of $UT_3(K)$ given by 
\[ \mathcal{A} =   K(e_{1,1} + e_{3,3}) \oplus Ke_{2,2} \oplus Ke_{1,2} \oplus Ke_{2,3} \oplus Ke_{1,3} , \]
where $e_{i,j}$ denote the matrix units, that is, $A$ consists of the matrices $\begin{pmatrix} d & a & c \\ 0 & g & b \\ 0 & 0 & d  \end{pmatrix}$. 

The polynomial identities of $\mathcal{A}$ were described by Gordienko in \cite{gordienko2009regev}. Furthermore, in \cite{pktcm},  \cite{koshlukov2013polynomial} the authors considered the generic algebra of $M_{1,1}(E)$ in two generators. It was shown in that paper that its polynomial identities are the same as the ones of $\mathcal{A}$. We are interested in the identities of the algebra $\mathcal{A}$ when considering additional structures (grading, involution, etc.), with the aim of determining a new PI-equivalence like in the case of traditional polynomial identities. Here we point out that the research in \cite{koshlukov2013polynomial} was motivated by a question posed by A. Berele, about the centre of the generic algebra of $M_{1,1}(E)$ in $d\ge 2$ generators. It is well known that the algebra generated by $d$ generic matrices for $M_n(K)$ is a (noncommutative) domain. Hence its centre is a commutative domain, and it can me embedded into its field of fractions. Several important questions in PI theory could be solved by using this simple trick.
The algebra $\mathcal{A}$, with involution, was also considered in \cite{MISHCHENKO200066}. 

We draw the readers' attention that in characteristic 0, the polynomial identities of $M_{1,1}(E)$ were determined by Popov \cite{popov}, and they follow from the identities $[[x_1,x_2]^2, x_1]$ and $[[x_1,x_2],[x_3,x_4], x_5]$. If one considers the generic algebra in two generators for $M_{1,1}(E)$ (or the algebra $\mathcal{A}$), the identities follow from the following three: $[[x_1,x_2][x_3,x_4],x_5]$, $[x_1,x_2][x_3,x_4][x_5,x_6]$, and from the standard identity $s_4$ of degree 4. Neither of these is an identity for $M_{1,1}(E)$, though. 

In this paper, we first study the gradings on the algebra $\mathcal{A}$, when the grading group is abelian. Additionally, we determine a basis for the $\mathbb{Z}_2$-graded identities of $\mathcal{A}$, and also for the identities with involution, and for the  $\mathbb{Z}_2$-graded identities with graded involution. We also determine the corresponding cocharacter sequence.

\section{Preliminaries}

\subsection{Graded identities}

Let $K$ be a fixed field of characteristic zero and $G$ a group. All algebras we consider will be associative and over $K$. 
The algebra $A$ is a $G$-graded algebra if $ A = \oplus_{g\in G} A_g $ is    a direct sum of the vector subspaces  $A_g$, satisfying  the relations
 $A_gA_h \subseteq A_{gh}$   for every $ g$, $h\in G$.  The subspaces $A_g$ are called the {\sl homogeneous components of $A$}, and the elements of each $A_g$ are called the {\sl homogeneous elements of $A$} of homogeneous degree $g$. We say that an associative algebra is a {\sl superalgebra} if $A$ is a $\mathbb{Z}_2$-graded algebra.

Let $G$ be a group, and consider the set of indeterminates  $X = \cup_{g\in G} X_g$, where $X_g = \{x_1^{(g)}, x_2^{(g)}, \dots \}$ is a countable infinite set for each $g$. We say that the indeterminates in $X_g$ have a homogeneous degree $g$, and the homogeneous degree of a monomial $x_{i_1}^{(g_{j_1})} \cdots x_{i_m}^{(g_{j_m})}  \in K\langle X \rangle$ is given by $g_{j_1} \cdots g_{j_m}$. We will write   $K\langle X \rangle^{gr}$ to denote the graded algebra $K\langle X \rangle$  with  $X = \cup_{g\in G} X_g$.

\begin{definition}
    Let $f(x_{i_1}^{(g_1)},x_{i_2}^{(g_2)}, \dots, x_{i_r}^{(g_r)}) \in K\langle X \rangle^{gr}$ be a polynomial. If $A = \oplus_{g\in G}A_g$ is a $G$-graded algebra then $f$ is a {\sl $G$-graded polynomial identity} (or simply a  {\sl $G$-graded identity}) for $A$ if $f(a_{i_1}^{(g_1)} , a_{i_2}^{(g_2)} , \dots , a_{i_r}^{(g_r)} ) = 0$ in $A$ for every homogeneous substitution $a^{(g_t)} \in  A_{g_t}$. We denote by $\Id_G(A)$ or $T_G(A)$ the ideal of all graded identities of $A$ in $K\langle X \rangle^{gr}$. 
\end{definition}
The ideal $\Id_G(A)$ is closed under all $G$-graded endomorphisms of $K\langle X \rangle^{gr}$; such ideals are called {\sl $G$-graded T-ideals}.

We are going to consider $G = \mathbb{Z}_2$. In this case we write $x_i^{(0)} = y_i$ and $x_i^{(1)} = z_i$, and also $K\langle X \rangle^{gr}$ as $K\langle Y,Z \rangle $. It is known that in characteristic zero, every graded polynomial identity is equivalent to a finite system of multilinear graded identities. Thus, given the free associative $\mathbb{Z}_2$-graded algebra $K \langle Y,Z \rangle$ we consider $P_{n_1,n_2}$  the space of multilinear polynomials in $n_1$ variables of degree $0$ and $n_2$ variables of degree $1$. 

For $n_1$, $n_2 \geq 0$, consider  the action of the group $S_{n_1}\times S_{n_2} $ on $P_{n_1,n_2}$ given by 
\[ 
	(\omega,\tau)f(y_{1},\dots,y_{n_1},z_{1},\dots,z_{n_2}) 
	= f(y_{\omega(1)},\dots,y_{\omega(n_1)},z_{\tau(1)},\dots,z_{\tau(n_2)}),
 \] 
where $(\omega, \tau) \in S_{n_1} \times S_{n_2} $ and $f \in P_{n_1,n_2}$. Then $P_{n_1,n_2}$ is a left $S_{n_1}\times S_{n_2}$-module.  

Given a $\mathbb{Z}_2$-graded algebra $A$, denote by $P_{n_1,n_2}(A)$ the quotient space \[ P_{n_1,n_2}(A) = \frac{ P_{n_1,n_2}}{ P_{n_1,n_2}\cap \Id_{\mathbb{Z}_2}(A)}.\] 
Since $T_2$-ideals are invariant under permutations of the variables, it follows that $P_{n_1,n_2} \cap  \Id_{\mathbb{Z}_2}(A)$ is an $S_{n_1} \times S_{n_2} $-submodule of $P_{n_1,n_2}$. Thus, $P_{n_1,n_2}(A)$  has a structure of an $S_{n_1} \times S_{n_2} $-module, and its character $\chi_{n_1,n_2}(A)$ is called the {\sl $(n_1,n_2)$-th graded cocharacter of $A$}. 

The $S_{n_1} \times S_{n_2}$-characters are obtained from the outer tensor product of irreducible characters of $S_{n_1}$ and $S_{n_2}$, and we have 1--1 correspondence between the irreducible characters of $S_{n_1} \times S_{n_2}$ and the pair of partitions $(\omega, \tau)$, where $\omega \vdash n_1$ and  $\tau \vdash n_2$. We denote by $\chi_{\omega} \otimes \chi_{\tau}$ the irreducible  $S_{n_1} \times S_{n_2}$-character associated to the pair of partitions $(\omega, \tau)$. Thus, 
\begin{equation}
	\chi_{n_1,n_2}(A) = \sum_{(\omega,\tau)\vdash (n_1,n_2)} m_{\omega,\tau} \chi_{\omega}\otimes  \chi_{\tau},
\end{equation}
where  $m_{\omega,\tau}$ is the multiplicity of   $\chi_{\omega}\otimes \chi_{\tau}$. The multiplicity $m_{\omega,\tau}$ is equal to the number of linearly independent highest weight vectors corresponding to a standard Young Tableau of shape $(\omega, \tau)$, modulo $\Id_{\mathbb{Z}_2}(A)$  (see \cite{drensky2000free} for more details).

\subsection{Identities with involution}

 The following type of polynomials identities that will be consider are the identities of an algebra endowed with an involution. To this end, let $X = \{x_1, x_2, \ldots \}$ be a countable set of non-commuting variables and consider $K\langle X, *\rangle = K\langle x_1, x_1^*, x_2, x_2^*, \ldots \rangle$, the {\sl free associative algebra with involution} in $X$ over $K$. By defining $y_i = x_i + x_i^*$ and $z_i = x_i - x_i^*$ for each $i = 1, 2, \ldots$, we consider $K\langle X, *\rangle = K\langle Y, Z \rangle = K\langle y_1, z_1, y_2, z_2, \ldots \rangle$ as generated by symmetric and skew-symmetric variables. The elements of $K\langle X, *\rangle$ will be called $*$-polynomials. 

 Note that we use $K\langle Y, Z \rangle$ to denote both the free $\mathbb{Z}_2$-graded algebra and the free algebra with involution. Therefore, the structure of the algebra $A$ will determine which free associative algebra is being considered. 
 
\begin{definition}
A $*$-polynomial $f\left( y_1,\dots,y_n,z_1,\dots,z_m \right) \in K\langle Y \cup Z \rangle$ is called a  {\sl $*$-polynomial identity} of an algebra with involution $(A,*)$ if \[ f\left( u_1,\dots,u_n,v_1,\dots,v_m \right) = 0 \text{ for all } u_i \in A^+ \text{ and } v_j \in A^- . \] 
\end{definition}
Given an algebra with involution $(A,*)$ we denote by $\Id^*(A)$ or $T^*(A)$ the set of $*$-polynomial identities of  $A$. 

In a similar way to the case of $\mathbb{Z}_2$-graded identities, we consider the vector space of multilinear polynomials in symmetric and skew-symmetric variables, that is, we consider $P_{n_1,n_2}$  the space of multilinear polynomials in $n_1$ symmetric variables   and $n_2$  skew-symmetric variables. Given $A$ an algebra with involution, we denote by $P_{n_1,n_2}(A)$ the quotient space \[ P_{n_1,n_2}(A) = \frac{ P_{n_1,n_2}}{ P_{n_1,n_2}\cap \Id^*(A)}.\] 

 Denote by $\Gamma$ the unitary subalgebra of $K\langle Y, Z\rangle$ generated by the elements from $Z$ and all non-trivial (long) commutators in the free variables of $K\langle Y, Z\rangle$. The elements of $\Gamma$  are called {\sl $Y$-proper polynomials}. We denote by $\Gamma_{n_1,n_2}$ the vector space
	\[  \Gamma_{n_1,n_2} = \Gamma \cap P_{n_1,n_2}. \] 

Given $A$ a unitary algebra with involution, we consider the spaces 	\[  \Gamma_{n_1,n_2}(A) =  \frac{ \Gamma_{n_1,n_2}}{ \Gamma_{n_1,n_2}\cap \Id^*(A)}.  \] 
Since $K$ is of characteristic zero, $\Id^*(A)$is generated, as a $T^*$-ideal, by its multilinear $Y$-proper polynomials.

\subsection{Graded involutions}

\begin{definition}
An involution $*$ on a $G$-graded algebra $A = \oplus_{g \in G} A_g$ is said to be a {\sl graded involution} if $A_g^*= A_g$ for all $g \in  G$.
\end{definition}
 
Similarly to the case of the free associative algebra with involution,  consider the free associative graded algebra with graded involution freely generated by homogeneous symmetric and homogeneous skew-symmetric variables.

Consider the free $\mathbb{Z}_2$-algebra with involution $K\langle Y \cup  Z, * \rangle$, generated by symmetric and skew elements of even and odd degree, that is \[ K\langle Y \cup  Z, * \rangle = K \langle y_{1}^{+}, y_{1}^{-}, z_{1}^{+}, z_{1}^{-}, y_{2}^{+}, y_{2}^{-}, z_{2}^{+}, z_{2}^{-}, \dots \rangle \]
where $y_{i}^{+}$ stands for a symmetric variable of even degree, $y_{i}^{-}$ for a skew variable of even degree, $z_{i}^{+}$ for a
symmetric variable of odd degree and $z_{i}^{-}$ for a skew variable of odd degree. 
Given $A$ a $\mathbb{Z}_2$-graded algebra with graded involution and a graded $*$-polynomial $f(y_{1}^{+},\dots,y_{m}^{+},y_{1}^{-},\dots,y_{n}^{-},z_{1}^{+},\dots,z_{s}^{+},z_{1}^{-},\dots,z_{t}^{-}) \in F\langle Y \cup  Z, * \rangle$, then $f$ is a {\sl graded $*$-polynomial identity} for $A$, if for all $u_{1}^{+},\dots,u_{m}^{+} \in A_{0}^{+}$, $u_{1}^{-},\dots,u_{n}^{-} \in A_{0}^{-}$, $v_{1}^{+},\dots,v_{s}^{+} \in A_{1}^{+}$, and  $v_{1}^{-},\dots,v_{t}^{-} \in A_{1}^{-}$, 
\[ f(u_{1}^{+},\dots,u_{m}^{+}, u_{1}^{-},\dots,u_{n}^{-}, v_{1}^{+},\dots,v_{s}^{+},v_{1}^{-},\dots,v_{t}^{-}) = 0. \]
We denote by $\Id_{\mathbb{Z}_2}^{*}(A) = \{ f \in K\langle Y \cup  Z\rangle  \mid f \equiv 0 \text{ on } A \}$ the $T_2^{*}$-ideal of graded $*$-identities of $A$, i.e., $\Id_{\mathbb{Z}_2}^{*}(A)$ is an ideal of $K\langle Y \cup  Z\rangle$, invariant under all $\mathbb{Z}_2$-graded endomorphisms of the free $\mathbb{Z}_2$-graded algebra commuting with the involution $*$. Again,  in characteristic zero, every graded  $*$-identity is equivalent to a finite system of multilinear graded $*$-identities. So, consider  $P_{n_1,n_2,n_3,n_4}$  the space of multilinear polynomials in $n_1$ symmetric variables of degree $0$, $n_2$ skew-symmetric variables of degree $0$,  $n_3$ symmetric variables of degree $1$   and $n_4$ skew-symmetric variables of degree $1$. 

We consider the cocharacters of $\mathbb{Z}_2$-graded algebras with graded involutions.  For $n_1$, $n_2$, $n_3$, $n_4 \geq 0$, consider  the action of the group $S_{n_1}\times S_{n_2} \times S_{n_3}  \times S_{n_4}$ on $P_{n_1,n_2, n_3, n_4}$ given by 
\[ \begin{split}
	(\omega,\sigma, \tau, \rho)f(y_{1}^{+},\dots,y_{n_1}^{+},y_{1}^{-},\dots,y_{n_2}^{-},z_{1}^{+},\dots,z_{n_3}^{+},z_{1}^{-},\dots,z_{n_4}^{-}) \\ 
	= f(y_{\omega(1)}^{+},\dots,y_{\omega(n_1)}^{+},y_{\sigma(1)}^{-},\dots,y_{\sigma(n_2)}^{-},z_{\tau(1)}^{+},\dots,z_{\tau(n_3)}^{+},z_{\rho(1)}^{-},\dots,z_{\rho(n_4)}^{-}),
\end{split}  \] 
where $(\omega,\sigma, \tau, \rho) \in S_{n_1} \times S_{n_2} \times S_{n_3} \times S_{n_4}$ and $f \in P_{n_1,n_2, n_3, n_4}$. 

Given a $\mathbb{Z}_2$-graded algebra $A$ with a graded involution, denote by $P_{n_1,n_2,n_3,n_4}(A)$ the quotient space \[ P_{n_1,n_2,n_3,n_4}(A) = \frac{ P_{n_1,n_2,n_3,n_4}}{ P_{n_1,n_2,n_3,n_4}\cap \Id_{\mathbb{Z}_2}^*(A)}.\] 
We examine the action of $S_{n_1} \times S_{n_2} \times S_{n_3} \times S_{n_4}$ on $P_{n_1,n_2,n_3,n_4}(A)$.  Since $T_2^*$-ideals are invariant under permutations of the variables, then $P_{n_1,n_2,n_3,n_4} \cap  \Id_{\mathbb{Z}_2}^*(A)$ is a left $S_{n_1} \times S_{n_2} \times S_{n_3} \times S_{n_4}$-submodule of $P_{n_1,n_2,n_3,n_4}$. Thus, $P_{n_1,n_2,n_3,n_4}(A)$  has a structure of left $S_{n_1} \times S_{n_2} \times S_{n_3} \times S_{n_4}$-module, and its character $\chi_{n_1,n_2,n_3, n_4}(A)$ is called the {\sl $(n_1,n_2,n_3,n_4)$-th cocharacter of $A$}. 

The $S_{n_1} \times S_{n_2} \times S_{n_3} \times S_{n_4}$-characters are obtained from the outer tensor product of irreducible characters of $S_{n_1}$, $S_{n_2}$, $S_{n_3}$ and $S_{n_4}$, and we have 1--1 correspondence between the irreducible characters of $S_{n_1} \times S_{n_2} \times S_{n_3} \times S_{n_4}$ and the 4-tuples of partitions $(\omega,\sigma, \tau, \rho)$, where $\omega \vdash n_1 $, $\sigma \vdash n_2$, $\tau \vdash n_3$ and $\rho \vdash n_4$. We denote by $\chi_{\omega}\otimes \chi_{\sigma} \otimes \chi_{\tau}\otimes \chi_{\rho}$ the irreducible  $S_{n_1} \times S_{n_2} \times S_{n_3} \times S_{n_4}$-character associated to the 4-tuple of partitions $(\omega,\sigma, \tau, \rho)$. Thus, 
\begin{equation}
	\chi_{n_1,n_2,n_3, n_4}(A) = \sum_{(\omega,\sigma,\tau,\rho)\vdash (n_1,n_2,n_3,n_4)} m_{\omega,\sigma,\tau,\rho} \chi_{\omega}\otimes \chi_{\sigma} \otimes \chi_{\tau}\otimes \chi_{\rho},
\end{equation}
where  $m_{\omega,\sigma,\tau,\rho}$ is the multiplicity of   $\chi_{\omega}\otimes \chi_{\sigma} \otimes \chi_{\tau}\otimes \chi_{\rho}$. 	The multiplicity $m_{\omega,\sigma,\tau,\rho}$ is equal to the number of linearly independent highest weight vectors corresponding to a standard Young Tableau of shape $(\omega,\sigma, \tau, \rho)$, modulo $\Id_{\mathbb{Z}_2}^*(A)$.

\section{Gradings on $A$}

We fix the algebra $ A =  K(e_{1,1} + e_{3,3}) \oplus Ke_{2,2} \oplus Ke_{1,2} \oplus Ke_{2,3} \oplus Ke_{1,3} $. 

We want to characterize the gradings on $A$ reducing them to elementary ones. As the main result of this section, we have that the gradings on the algebra $A$ are equivalent to elementary gradings.

\begin{definition}
	Given a triple $\hat{g} = (g_1,g_2,g_3)$ of elements of an arbitrary group  $G$ and $1\leq i,j \leq 3$,  we set the degrees for the unitary matrices as $\deg e_{i,j} = g^{-1}_{i} g_{j}$ for $i \neq j$ or $i=j=2$ , and $\deg (e_{1,1} + e_{3,3}) = 1_G$. Given $g \in G$ let $A_g = \Span\{e_{i,j} \mid g^{-1}_{i}g_{j}= g\}$ for $g \neq 1_G$ and $A_{1_G} = \Span\left( \{e_{i,j} \mid g^{-1}_{i}g_{j}= 1_G\} \cup \{ e_{1,1} + e_{3,3} \} \right)$. Then $A = \oplus_{g\in G} A_g$ is a $G$-grading of $A$ called the {\sl  elementary} $G$-grading defined by the $3$-tuple $\hat{g}$.
\end{definition}

We start considering the idempotents of the algebra $A$. 

\begin{proposition}
\label{idempotent}
If $e \in A$ is an idempotent element, then $e$ is conjugated to a diagonal element of $A$. 
\end{proposition}

\begin{proof}
    If $e = \begin{pmatrix} x & a & c \\ 0 & y & b \\ 0 & 0 & x  \end{pmatrix} \neq \mathbf{0} $ and $e^2 = e$ then  $ \begin{pmatrix} x^2 & a(x+y) & 2xc + ab \\ 0 & y^2 & b(x+y) \\ 0 & 0 & x^2  \end{pmatrix} =  \begin{pmatrix} x & a & c \\ 0 & y & b \\ 0 & 0 & x  \end{pmatrix}$ which implies that   $x,y \in \{0,1\}$ where $x$ and $y$ are not simultaneously zero.  So, we consider the following cases:

\begin{itemize}
    \item If $x = 0$, $y=1$ then 
        $e = \begin{pmatrix} 0 & a & ab \\ 0 & 1 & b \\ 0 & 0 & 0  \end{pmatrix}$. If $q = \begin{pmatrix} 1 & -a & 0 \\ 0 & 1 & b \\ 0 & 0 & 1  \end{pmatrix}$ a simple computation shows that $qeq^{-1} = \begin{pmatrix} 0 & 0 & 0 \\ 0 & 1 & 0 \\ 0 & 0 & 0  \end{pmatrix}$.

    \item If $x = 1$, $y=0$ then 
        $e = \begin{pmatrix} 1 & a & -ab \\ 0 & 0 & b \\ 0 & 0 & 1  \end{pmatrix}$. Choose $q = \begin{pmatrix} 1 & a & 0 \\ 0 & 1 & -b \\ 0 & 0 & 1  \end{pmatrix}$, then $qeq^{-1} = \begin{pmatrix} 1 & 0 & 0 \\ 0 & 0 & 0 \\ 0 & 0 & 1  \end{pmatrix}$.

    \item If $x = 1$, $y=1$ then 
        $e = \begin{pmatrix} 1 & 0 & 0 \\ 0 & 1 & 0 \\ 0 & 0 & 1  \end{pmatrix}$.
\end{itemize}
\end{proof}

\begin{remark} \label{rm:conjg}
   Note that if we have two orthogonal idempotent elements, then each one of these is conjugated to an element of the set $\{ (e_{1,1} + e_{3,3}), e_{2,2} \}$. 
\end{remark}

\begin{remark} 
Let $A = \oplus_{g \in G} A_g$ be graded by an abelian group. Since for $g$, $h \in G$, $[A_g , A_h] \subset A_{gh} + A_{hg} \subset A_{gh}$, it follows that the commutator subalgebra $[A, A]$ is a nilpotent non-zero graded ideal of A.  
\end{remark}

\begin{lemma} \label{lm:orthInden}
    Let $A = \oplus_{g \in G} A_g$ be graded by an abelian group $G$ with identity element $1_G \in G$. Then $A_{1_G}$ contains $2$ orthogonal idempotents.
\end{lemma}

\begin{proof} 
Let $E$ be the identity element of $A$. Since $E \in A_{1_G}$, there exists a non trivial maximal semisimple subalgebra $B$ of $A_{1_G}$.
Let $C$ be one of the simple summands of $B$ and let $e$ be its unit element. 

By Proposition \ref{idempotent} it follows that $e$ is conjugated to a diagonal idempotent. Hence either $e$ and $E - e$ are two orthogonal idempotents or $e = E$ and $C = B = \Span{E}$. 

Consider the case $e = E$ and $B = \Span{E}$, we shall show that this leads to a contradiction and hence is impossible. 

First, note that  any homogeneous element of $A$ is either nilpotent or invertible. 

In fact, suppose that $a \in  A_g$ is not nilpotent. For $m$ large enough the elements $a$, $a^2$,\dots, $a^m$ are linearly dependent and homogeneous. It follows that $g$ must have finite order $k$ and $a^k \in A_{1_G}$. Moreover being not nilpotent the element $a^k$ does not lie in the Jacobson radical $J(A_{1_G})$ of $A_{1_G}$. Since $A_{1_G}/J(A_{1_G}) \cong K\cdot E$, it follows that $a^k - \lambda E$ is nilpotent for some $\lambda \in K$, $\lambda \neq 0$. This means that $a^k$ and hence also $a$ is invertible.

Next we prove that the Jacobson radical $J(A)$ of $A$, i.e. the subalgebra of all
strictly upper-triangular matrices does not contain non-zero homogeneous elements. 
 $J(A)$ is homogeneous in the $G$-grading since $\varphi(J(A)) = J(A)$ for any $\varphi  \in  {\rm Aut }A$. 

In fact suppose by contradiction that $0 \neq a \in A_g$ is nilpotent and consider its left annihilator
$L_a = \{x \in  A \mid xa = 0\}$ and right  annihilator $R_a = \{x \in  A \mid ax = 0\}$, these are graded subspaces of $A$. Then, as the elements of $L_a$ and $R_a$ are zero divisors they are not invertible, hence they are nilpotent. By our hypothesis  $a$ is nilpotent, $a = \begin{pmatrix} 0 & a_1 & a_3 \\ 0 & 0 & a_2 \\ 0 & 0 & 0  \end{pmatrix}$ and $a^2 = \begin{pmatrix} 0 & 0 & a_1 a_2 \\ 0 & 0 & 0 \\ 0 & 0 & 0  \end{pmatrix}$. Note that, if $a^2 \neq 0$, $a^2 \in A_{g^2}$ and then $e_{2,2} \in L_{a^2}$, a contradiction. Now, if $a^2 = 0$ then $a = \begin{pmatrix} 0 & a_1 & a_3 \\ 0 & 0 & 0 \\ 0 & 0 & 0  \end{pmatrix}$ or $a = \begin{pmatrix} 0 & 0 & a_3 \\ 0 & 0 & a_2 \\ 0 & 0 & 0  \end{pmatrix}$, so $e_{2,2} \in L_a$ or $e_{2,2} \in R_a$, once again a contradiction.  Therefore, $J(A)$ has no nonzero homogeneous elements. But this gives us a contradiction because $[A, A] \subset J(A)$.  

Thus, we have $e$ and $E - e$ are two orthogonal idempotents belonging to $A_{1_G}$. 
\end{proof}

\begin{lemma} \label{lm:homog}
    Let $A = \oplus_{g \in G} A_g$ be $G$-graded. Then the grading is elementary if and only if all matrix units $e_{i,j}$, $1\leq i < j \leq 3$, $e_{2,2}$ and $(e_{1,1} + e_{3,3})$ are homogeneous. 
\end{lemma}
\begin{proof}
    If the $G$-grading is  elementary then those matrices are homogeneous by definition. Suppose that the said matrices are homogeneous. If we set $g_1= 1_G$, $g_2 = \deg e_{1,2}$ and $g_3 = g_2\deg e_{2,3}$ then the triple $(g_1,g_2,g_3)$ satisfies the conditions for the grading to be elementary.
\end{proof}

\begin{lemma} \label{lm:lm}
    Let $A = \oplus_{g \in G} A_g$ be $G$-graded. Then the grading is  elementary if and only if the matrices $e_{2,2}$ and $(e_{1,1} + e_{3,3})$ belong to $A_{1_G}$. 
\end{lemma}
\begin{proof}
    It is clear that if the grading is  elementary then $e_{2,2}$, $(e_{1,1} + e_{3,3}) \in A_{1_G}$. 

    Now, assume that  $e_{2,2}$, $(e_{1,1} + e_{3,3}) \in A_{1_G}$ and note that 
    \begin{itemize}
        \item $e_{1,2} \in (e_{1,1} + e_{3,3}) A e_{2,2}$ with $\dim  (e_{1,1} + e_{3,3}) A e_{2,2}= 1$,

        \item $e_{2,3} \in e_{2,2} A (e_{1,1} + e_{3,3})$ with $\dim  e_{2,2} A (e_{1,1} + e_{3,3})= 1$,
        
        \item $e_{1,3} \in (e_{1,1} + e_{3,3}) A (e_{1,1} + e_{3,3}) = K(e_{1,1} + e_{3,3}) \oplus K e_{1,3}$. 
   \end{itemize}
    Since $(e_{1,1} + e_{3,3}) A e_{2,2}$, $(e_{1,1} + e_{3,3}) A (e_{1,1} + e_{3,3})$ and $e_{2,2} A (e_{1,1} + e_{3,3})$ are graded subalgebras then we can conclude that the elements $e_{1,2}$, $e_{1,3}$ and $e_{1,2}$ are homogeneous. Thus, by Lemma \ref{lm:homog} the grading is  elementary.
\end{proof}

\begin{lemma} \label{lm:idempDiago}
    Let $A = \oplus_{g \in G} A_g$ be $G$-graded. Then there exist two orthogonal idempotents that are simultaneously diagonalizable and belonging to $A_{1_G}$. 
\end{lemma}

\begin{proof}
    Suppose $A$ is $G$-graded, then the identity matrix $E$ is homogeneous. 
    
    Also $J$, the Jacobson radical of $A$ is homogeneous. So is $J^2=\Span(e_{13})$. By the Wedderburn-Malcev theorem, $A\cong K+K+J$, a direct sum of vector spaces, where $K+K$ is a subalgebra. Hence we have two orthogonal idempotents. We show we can choose these homogeneous and simultaneously diagonalizable. 

    First take the intersection of the left and right annihilators of $J^2$, it is homogeneous. But this is exactly the span $\Span(e_{1,2},e_{1,3},e_{2,3},e_{2,2})$. There exists an idempotent in it, moreover all idempotents in it are of the form $\begin{pmatrix}0&a&ac\\ 0&1&c\\ 0&0&0 \end{pmatrix}$. 

    Suppose the element above is the homogeneous idempotent, call it $t$. 
    Now, considering $A$ acting on a $3$-dimensional vector space with a basis $\{e_1, e_2, e_3\}$, we obtain 
    \[
    te_1=0, \quad te_2=ae_1+e_2, \quad te_3=ace_1+ce_2.
    \]
    Now consider the basis $f_1$, $f_2$, $f_3$ such that $tf_1=0$, $tf_2=f_2$, $tf_3=0$. Such a basis can be obtained as
    \[
    f_1=e_1, \quad f_2=ae_1+e_2, \quad f_3= -ce_2+e_3.
    \]
    ($f_1$ and $f_3$ are a basis of the kernel of $t$, and $f_2$ of its image.)
    The change of basis (matrix) is $P=\begin{pmatrix} 1&a&0\\ 0&1&-c\\
    0&0&1\end{pmatrix}$.  
    Form $I-t$, this is an idempotent which is orthogonal to $t$, and clearly the same $P$ diagonalizes it. As $t$ is homogeneous and $I$ also is, then $I-t$ is homogeneous. 
\end{proof}

\begin{theorem}\label{th:gradingofA}
    Let $G$ be an abelian group and $K$ a field. Suppose that the $K$-algebra $A = \oplus_{g \in G} A_g$ is $G$-graded. Then $A$, as a $G$-graded algebra, is isomorphic to $A$ with an  elementary $G$-grading.
\end{theorem}

\begin{proof}
    Let $A = \oplus_{g \in G} A_g$ be $G$-graded. By Lemma \ref{lm:idempDiago} the homogeneous component $A_{1_G}$ contains two orthogonal idempotents simultaneously diagonalizable by a element $P$.  So as a $G$-graded algebra $A$ is isomorphic to $A' = \oplus_{g\in G} A'_g$ where $A'_g = P^{-1}A_g P$. Note that in $A'$ the matrix $e_{2,2}$ and the matrix identity $E$ lie in $A'_1$. Therefore, by Lemma \ref{lm:lm}, $A'$ has an  elementary $G$-grading.
\end{proof}

Next, we will consider the graded identities of the algebra $A$, where the grading is given by the group $\mathbb{Z}_2$. 
Due to Theorem \ref{th:gradingofA}, we have $3$ different non-trivial $\mathbb{Z}_2$-gradings on  $A$. Thus, for $A= A_0 \oplus A_1$ we have the possibilities: 
\begin{equation}\label{eq:gra1}
    A_0 = \left\{  
\begin{pmatrix} 
d & 0 & c \\ 
0 & g & 0 \\ 
0 & 0 & d  
\end{pmatrix} \mid d,g,c \in K  \right\}, \qquad 
A_1 = \left\{  \begin{pmatrix} 
0 & a & 0 \\ 
0 & 0 & b \\ 
0 & 0 & 0  
\end{pmatrix} \mid a,b \in K  \right\},
\end{equation}
or 
\begin{equation}\label{eq:gra2}
    A_0 = \left\{  
\begin{pmatrix} 
d & a & 0 \\ 
0 & g & 0 \\ 
0 & 0 & d  
\end{pmatrix} \mid d,g,a \in K  \right\}, \qquad 
A_1 = \left\{  \begin{pmatrix} 
0 & 0 & c \\ 
0 & 0 & b \\ 
0 & 0 & 0  
\end{pmatrix} \mid c, b \in K  \right\},
\end{equation}
or
\begin{equation}\label{eq:gra3}
	A_0 = \left\{  
	\begin{pmatrix} 
		d & 0 & 0 \\ 
		0 & g & b \\ 
		0 & 0 & d  
	\end{pmatrix} \mid d,g,b \in K \right\}, \qquad 
	A_1 = \left\{  \begin{pmatrix} 
		0 & a & c \\ 
		0 & 0 & 0 \\ 
		0 & 0 & 0  
	\end{pmatrix} \mid a,c \in K  \right\}. 
\end{equation}
Denote  by $\mathcal{A}^1$ the graded algebra $A$ with grading given by (\ref{eq:gra1}),   $\mathcal{A}^2$ the graded algebra $A$ with grading given by (\ref{eq:gra2}), and $\mathcal{A}^3$ the graded algebra $A$ with grading given by (\ref{eq:gra3}).

\section{Graded polynomial identities of $\mathcal{A}^1$}

It is easy to see that the followings are graded identities of $\mathcal{A}^1$. 

\begin{itemize}
    \item[(a)] $[y_1, y_2] = 0$,
    \item[(b)] $z_1z_2z_3 = 0$,
    \item[(c)] $[y_1,z_1z_2] = 0$.  
\end{itemize}
	The identity (c) is a direct consequence of identity (a).

   Note that $P_{m,n}(\mathcal{A}^1) = 0$ if $n\geq 3$. So, we consider the cases when $n=0,1,2$.  In the case $n=0$, by the identity (a) it is clear that 
   \begin{equation} \label{eq:A_1n=0}
        P_{m,0}(\mathcal{A}^1) = \Span\{y_1y_2\cdots y_m\}.  
   \end{equation}

\subsection{Case $n=1$}
 Given a monomial $\mu$ in $P_{m,1}(\mathcal{A}^1)$, from (a) we can reorder the variables $y_i$ such that modulo $\Id_{\mathbb{Z}_2}(\mathcal{A}^1)$  \begin{equation} \label{eq:A_1n=1}
     \mu = y_{i_1}y_{i_2}\cdots y_{i_s}z_1 y_{j_1}y_{j_2}\cdots y_{j_t} 
 \end{equation} where $ i_1 < i_2 \cdots < i_s$, $ j_1 < j_2 \cdots < j_t$ and $s+t = m$, $s,t \geq 0$. 

\begin{proposition}\label{pro:A_1n=1}
 The monomials of the form (\ref{eq:A_1n=1}) are linearly independent modulo $\Id_{\mathbb{Z}_2}(\mathcal{A}^1)$.
\end{proposition}
\begin{proof}
Let $f$ be a sum of monomials of the form (\ref{eq:A_1n=1}) \[ f = \sum_{I,J}\alpha_{I,J}y_{i_1}y_{i_2}\cdots y_{i_s}z_1 y_{j_1}y_{j_2}\cdots y_{j_t}.\]
Note that if $y_{i_k} = \beta_{i_k}(e_{1,1} + e_{3,3})$, $z_1 = \lambda e_{1,2}$ and $y_{j_l} = \eta_{j_l}e_{2,2}$ then \[ f = \sum_{I,J}\alpha_{I,J}\beta_{i_1}\beta_{i_2}\cdots \beta_{i_s}\lambda \eta_{j_1}\eta_{j_2}\cdots \eta_{j_t} \cdot e_{1,2}.\]
Suppose that $f$ is a polynomial identity of $\mathcal{A}^1$ and that there exist $\alpha_{I_0,J_0} \neq 0$. Consider the evaluation $y_{i_k} = e_{1,1} + e_{3,3}$ for $i_{k} \in I_0$, $y_{j_k} = e_{2,2}$ for $j_{k} \in J_0$ and $z_1= e_{1,2}$, then $\alpha_{I_0,J_0} = 0$, a contradiction. So,  the monomials of the form (\ref{eq:A_1n=1}) are linearly independent. 
\end{proof}

\subsection{Case $n=2$}

 Given a monomial $\mu$ in $P_{m,2}(\mathcal{A}^1)$, again from (a) we can reorder the variables $y_i$ such that modulo $\Id_{\mathbb{Z}_2}(\mathcal{A}^1)$  \begin{equation*} 
     \mu = y_{i_1}y_{i_2}\cdots y_{i_s}z_{l_1} y_{j_1}y_{j_2}\cdots y_{j_t} z_{l_2} y_{h_1}y_{h_2}\cdots y_{h_r}
 \end{equation*} where $ i_1 < i_2 \cdots < i_s$, $ j_1 < j_2 \cdots < j_t$, $ h_1 < h_2 \cdots < h_r$  and $s+t+r = m$, $s,t,r \geq 0$.  
As a consequence of (c) we have $z_{l_1}y_{j}z_{l_2}y_h = y_h z_{l_1}y_{j}z_{l_2} $. Therefore, $P_{m,2}(\mathcal{A}^1)$ is spanned by monomials of the form  
\begin{equation} \label{eq:A_1n=2}
     y_{i_1}y_{i_2}\cdots y_{i_s}z_{l_1} y_{j_1}y_{j_2}\cdots y_{j_t} z_{l_2}
 \end{equation} 
 where $ i_1 < i_2 \cdots < i_s$, $ j_1 < j_2 \cdots < j_t$,   and $s+t = m$, $s,t \geq 0$.  

 \begin{proposition} \label{pro:A_1n=2}
      The monomials of the form (\ref{eq:A_1n=2}) are linearly independent modulo $\Id_{\mathbb{Z}_2}(\mathcal{A}^1)$.
 \end{proposition}

\begin{proof}
    Note that if $ y_{h_k} = \begin{pmatrix} a_{h_k} & 0 & c_{h_k} \\
    0 & b_{h_k} & 0 \\
    0 & 0 & a_{h_k} 
    \end{pmatrix} $ and  $ z_{l_i} = \begin{pmatrix} 0 & d_{l_i} & 0 \\
    0 & 0 & e_{l_i} \\
    0 & 0 & 0 
    \end{pmatrix} $ then \[ y_{i_1}y_{i_2}\cdots y_{i_s}z_{l_1} y_{j_1}y_{i_2}\cdots y_{j_t} z_{l_2} = a_{i_1}\cdots a_{i_s} d_{l_1} b_{j_1}\cdots b_{j_t}d_{l_2} \cdot e_{1,3}. \]
    Let $f$ be a sum of monomials of the form (\ref{eq:A_1n=2}) 
    \begin{align*} f &= \sum_{I,J,L}\alpha_{I,J,L}y_{i_1}y_{i_2}\cdots y_{i_s}z_{l_1} y_{j_1}y_{j_2}\cdots y_{j_t}z_{l_2}\\ 
& = \sum_{I,J,L}\alpha_{I,J,L}a_{i_1}\cdots a_{i_s} d_{l_1} b_{j_1}\cdots b_{j_t}e_{l_2} \cdot e_{1,3}.
\end{align*}
    Suppose $f$ is a polynomial identity of $\mathcal{A}^1$ and  there exists $\alpha_{I_0,J_0,L_0} \neq 0$. Consider the evaluation $y_{i_k} = e_{1,1} + e_{3,3}$ for $i_{k} \in I_0$, $y_{j_k} = e_{2,2}$ for $j_{k} \in J_0$, $z_{l_1}= e_{1,2}$, and $z_{l_2}= e_{2,3}$. Then, $\alpha_{I_0,J_0,L_0} = 0$, a contradiction. So,  the monomials in the form (\ref{eq:A_1n=2}) are linearly independent. 
\end{proof}

As a consequence of (\ref{eq:A_1n=0}), Proposition \ref{pro:A_1n=1} and Proposition \ref{pro:A_1n=2}, we have the following result. 

\begin{theorem} \label{th:A_1}
The identities (a) and (b),  form a basis for the $\mathbb{Z}_2$-graded identities for the algebra $\mathcal{A}^1$.
\end{theorem}

\section{Graded polynomial identities of $\mathcal{A}^2$}

Modulo $\Id_{\mathbb{Z}_2}(\mathcal{A}^2)$ we have that \begin{equation} \label{eq:m1-m4}
    z_{j_1}z_{j_2} = 0 \text{ and } [y_{i_1}, y_{i_2}][y_{i_3}, y_{i_4}] = 0
\end{equation}
Once again, we consider the space $P_{m,n}$ of multilinear polynomials in the variables $y_{i_1}, y_{i_2}, \dots, y_{i_m}$, $z_{j_1}, z_{j_2}, \dots, z_{j_n}$. Since $ z_{j_1}z_{j_2} = 0$, we have $P_{m,n} = 0 $ if $n\geq 2$. Thus we consider the cases $n=0$ and $n=1$. 

\subsection{Case $n=0$ }

By (\ref{eq:m1-m4}) we have that $ [ y_{i_1}, y_{i_2}][y_{i_3}, y_{i_4}] = 0$, then considering proper polynomials, we have only sums of commutators. Also from (\ref{eq:m1-m4}) and since  \[ [[y_{i_1}, y_{i_2}],[y_{i_3}, y_{i_4}]] = [y_{i_1}, y_{i_2}, y_{i_3}, y_{i_4}] - [y_{i_1}, y_{i_2}, y_{i_4}, y_{i_3}], \] given a commutator we can reorder the variables in the way $[  y_{k}, y_{l} , y_{i_1},  y_{i_2},\dots,  y_{i_{m-2}}]$ where $i_1 < i_2 < \cdots < i_{m-2}$, and by Jacobi identity we can only consider the case $[ y_{i_1},  y_{i_2},\dots,  y_{i_{m}} ]$ where $i_1 > i_2 < i_3 < \cdots < i_{m-2}$. 

Also, note that the subalgebra $\mathcal{A}^2_0$ is isomorphic to the algebra $UT_2(K)$. So, the polynomial $ [ y_{i_1}, y_{i_2}][y_{i_3}, y_{i_4}]$ is a basis of the identities in the variables $y_{i}$.

\subsection{Case $n=1$}

In this case, we consider once again proper polynomials, that is, sums of products of left normed commutators.  

Note that since $ [y_{i_1},y_{i_2}] [y_{i_3},y_{i_4}] = 0$ and we have only one $z_j$ then we can consider sums of commutators with products of two commutators. 

 Now, since $ [y_{i_1},y_{i_2}] = \nu   e_{1,2}$, then  \begin{equation}\label{eq:N2}
     z_{j_1}[y_{i_1},y_{i_2}] = 0. 
 \end{equation}
hence in the case when we have products of two commutators we get that $z_{j_1}$ lies in the second commutator. 

By direct computation, we also have the following identity: 
\begin{equation} \label{eq:N3}
    [[y_{i_1},y_{i_2}]z_{j_1}, y_{i_3}] = 0,
\end{equation} 
Note that from (\ref{eq:N2}) $ [y_{i_1},y_{i_2},z_{j_1}] = [y_{i_1},y_{i_2}]z_{j_1}$  and using the identity (\ref{eq:N3}) we have $  [y_{i_1},y_{i_2},z_{j_1},y_{i_3}] = 0$. Therefore, if $z_{j_1}$ lies in a commutator $\omega$, that commutator has the form $\omega =  [z_{j_1},y_{i_1},y_{i_2}, \dots,y_{i_k}] $ or $\omega =  [y_{i_1},y_{i_2}, \dots,y_{i_k},z_{j_1}] = [y_{i_1},y_{i_2}, \dots,y_{i_k}]z_{j_1} $.

Let us consider the case $\omega =  [z_{j_1},y_{i_1},y_{i_2}, \dots,y_{i_k}] $. Note that by Jacobi identity $ [z_{j_1},y_{i_1},y_{i_2}] + [y_{i_1},y_{i_2},z_{j_1}] + [y_{i_2},z_{j_1},y_{i_1}] = 0 $, then 
$ [z_{j_1},y_{i_2},y_{i_1}] = [z_{j_1},y_{i_1},y_{i_2}] + [y_{i_1},y_{i_2}]z_{j_1}$. So we can write the variables $y_i$ in any order in $\omega$. 

If  $\omega =  [y_{i_1},y_{i_2}, \dots,y_{i_k},z_{j_1}] = [y_{i_1},y_{i_2}, \dots,y_{i_k}]z_{j_1} $, as in the case $P_{m,0}$, we can reorder the variables $y_{i_1},y_{i_2}, \dots,y_{i_k}$ such that  $[y_{i_1},y_{i_2}, \dots,y_{i_k}]z_{j_1}$  is a linear combination of polynomials $[y_{i_l},y_{i_1},y_{i_2}, \dots ,\widehat{y_{i_l}}, \dots,y_{i_s}]z_{j_1}$  where  $i_{l} > i_1 < i_2 < \cdots < i_k$.

At this point, we have that if $\mathcal{B}$ is the space of proper polynomials in the variables $z_{j_1}$, $y_{i_1}$, $y_{i_2}$, \dots, $y_{i_k}$, modulo $\Id_{\mathbb{Z}_2}(\mathcal{A}^2)$, $\mathcal{B}$ is spanned by \begin{itemize}
    \item $ [z_{j_1},y_{i_1},y_{i_2}, \dots,y_{i_k}]$, $i_1 < i_2 < \cdots < i_k$,
    \item $ [y_{i_l},y_{i_1},y_{i_2}, \dots ,\widehat{y_{i_l}}, \dots,y_{i_k}]z_{j_1}$, $i_{l} > i_1 < i_2 < \cdots < i_k$,
    \item $ [y_{i_l},y_{i_1},y_{i_2}, \dots ,\widehat{y_{i_l}}, \dots,y_{i_s}][z_{j_1}, y_{h_1},y_{h_2},\dots,y_{h_t} ]$,  $i_{l} > i_1 < i_2 < \cdots < i_k$, $h_1 < h_2 < \cdots < h_s$.
\end{itemize}

\begin{proposition} \label{pr:A_2}
    The polynomial \begin{equation}
        [y_{i_1},y_{i_2}][z_{j_1},y_{i_3}] + [y_{i_1},y_{i_2},y_{i_3}]z_{j_1}
    \end{equation} 
    is a polynomial identity of the algebra $\mathcal{A}^2$.
\end{proposition}
\begin{proof}
    Note that  \begin{align*}
        [y_{i_1},y_{i_2}][z_{j_1},y_{i_3}] &  = [y_{i_1},y_{i_2}]z_{j_1}y_{i_3} -   [y_{i_1},y_{i_2}]y_{i_3}z_{j_1} \stackrel{(\ref{eq:N3})}{=} y_{i_3}[y_{i_1},y_{i_2}]z_{j_1} -   [y_{i_1},y_{i_2}]y_{i_3}z_{j_1} \\
         & = [y_{i_3},[y_{i_1},y_{i_2}]]z_{j_1}  = -[y_{i_1},y_{i_2},y_{i_3}]z_{j_1}.\qedhere
    \end{align*}
\end{proof}

Hence from Proposition \ref{pr:A_2} we consider linear combinations of 
\begin{itemize}
    \item $ [z_{j_1},y_{i_1},y_{i_2}, \dots,y_{i_k}]$, $i_1 < i_2 < \cdots < i_k$,
    \item $ [y_{i_l},y_{i_1},y_{i_2}, \dots ,\widehat{y_{i_l}}, \dots,y_{i_k}]z_{j_1}$, $i_{l} > i_1 < i_2 < \cdots < i_k$.
\end{itemize}
Note that, if 
$ y_{i_k} = \begin{pmatrix} 
    d_{i_k} & a_{i_k} & 0 \\
    0 & g_{i_k} & 0 \\
    0 & 0 & d_{i_k} 
    \end{pmatrix} $ 
    and  
    $ z_{j_1} = \begin{pmatrix} 0 & 0 & c_{j_1} \\
    0 & 0 & b_{j_1} \\
    0 & 0 & 0 
    \end{pmatrix} $, then 
     \[
        [z_{j_1},y_{i_1},y_{i_2}, \dots,y_{i_m}]  =   (-1)^{m+1}b_{j_1} \prod_{k=1}^{m-1} (d_{i_k} - g_{i_k})a_{i_m} \cdot e_{1,3}   + b_{j_1} \prod_{k=1}^{m} (d_{i_k} - g_{i_k}) \cdot e_{2,3}
\]
 and 
\[
      [y_{i_l},y_{i_1},y_{i_2}, \dots ,\widehat{y_{i_l}}, \dots,y_{i_k}]z_{j_1}    
     =  \{ 
     \left[ a_{i_1}(d_{i_l} - g_{i_l}) + a_{i_l}(g_{i_l} - d_{i_l}) \right] \prod_{k \neq l}^{m} (g_{i_k} - d_{i_k})
     \} b_{j_1} \cdot e_{2,3}
 \]
\begin{theorem}\label{th:A_2}
The following identities  form a basis for the $\mathbb{Z}_2$-graded identities of the algebra $\mathcal{A}^2$: 
\begin{itemize}
    \item[$(i)$] $ z_{i_1}z_{i_2}$,
    \item[$(ii)$] $ [y_{i_1},y_{i_2}][y_{i_3},y_{i_4}] $,
    \item[$(iii)$]  $ z_{i_1}[y_{i_1},y_{i_2}] $,
    \item[$(iv)$]  $ [[y_{i_1},y_{i_2}]z_{j_1},y_{i_3}]$.
\end{itemize}
\end{theorem}

\begin{proof}
    We want to prove that the polynomials $ [z_{j_1},y_{i_1},y_{i_2}, \dots,y_{i_m}]$ with  $i_1 < i_2 < \cdots < i_m$, and $ [y_{h_l},y_{h_1},y_{h_2}, \dots ,\widehat{y_{h_l}}, \dots,y_{h_m}]z_{j_1} $ with $h_{l} > h_1 < h_2 < \cdots < h_m$ are linearly independent modulo $\Id_{\mathbb{Z}_2}(\mathcal{A}^2)$. Consider $f$ the polynomial given by \[ f = \alpha[z_{j_1},y_{i_1},y_{i_2}, \dots,y_{i_m}] + \sum_{l=2}^{m}\beta_l [y_{h_l},y_{h_1},y_{h_2}, \dots ,\widehat{y_{h_l}}, \dots,y_{h_m}]z_{j_1} \] and suppose that $f\in \Id_{\mathbb{Z}_2}(\mathcal{A}^2)$. Now, considering the evaluation $z_{j_1} = e_{2,3}$ and $y_{i} = (e_{1,1} + e_{3,3})$ for $i = 1$, \dots, $m$ one has $\alpha e_{2,3} = 0$, so $\alpha = 0$. Then, \[ f = \sum_{l=2}^{m} \beta_l [y_{h_l},y_{h_1},y_{h_2}, \dots ,\widehat{y_{h_l}}, \dots,y_{h_m}]z_{j_1}=0.\] 
    Suppose that there exists $l_0$ such that $\beta_{l_0} \neq 0$, and consider the evaluation $y_{h_{l_0}} = e_{1,2}$, $y_{h_i} = e_{2,2}$ for $i \neq l_0$ and $z_{j_1} = e_{2,3}$. Then, $\beta_{l_0}e_{1,3} = 0$, therefore $\beta_{l_0} = 0$, a contradiction. Thus, we have the desired result. 
\end{proof}

\section{  Graded polynomial identities of $\mathcal{A}^3$}

By direct computation, the following polynomials are identities of $\mathcal{A}^3$:
\begin{itemize}
    \item[(a)] $ z_{j_1}z_{j_2} $,
    \item[(b)] $ [y_{i_1},y_{i_2}][y_{i_3},y_{i_4}]$,
    \item[(c)] $ [y_{i_1},y_{i_2}]z_{j_1}$,
    \item[(d)] $ [z_{j_1}[y_{i_1},y_{i_2}],y_{i_3}]$. 
\end{itemize}

\begin{proposition} \label{pr:A_3}
    The polynomial \begin{equation}
        [z_{j_1},y_{i_3}][y_{i_1},y_{i_2}] + z_{j_1}[y_{i_1},y_{i_2},y_{i_3}]
    \end{equation} 
    is a polynomial identity of the algebra $\mathcal{A}^3$.
\end{proposition}
\begin{proof}
    Note that  \begin{align*}
       [z_{j_1},y_{i_3}] [y_{i_1},y_{i_2}] &  = z_{j_1}y_{i_3}[y_{i_1},y_{i_2}] -   y_{i_3}z_{j_1} [y_{i_1},y_{i_2}]  
\stackrel{(d)}{=}z_{j_1}y_{i_3}[y_{i_1},y_{i_2}] -   z_{j_1} [y_{i_1},y_{i_2}]y_{i_3} \\
         & = z_{j_1}[y_{i_3},[y_{i_1},y_{i_2}]]  = -z_{j_1}[y_{i_1},y_{i_2},y_{i_3}].\qedhere
    \end{align*}
\end{proof}
Consider the vector space $P_{m,n}$ of multilinear polynomials in the variables $y_{i_1}$, $y_{i_2}$, \dots, $y_{i_m}$, $z_{j_1}, z_{j_2}, \dots, z_{j_n}$. Since $ z_{j_1}z_{j_2} = 0$, $P_{m,n} = 0 $ if $n\geq 2$. So, we consider the cases $n=0$ and $n=1$.

\subsection{Case $n=0$}

Since $[y_{i_1},y_{i_2}][y_{i_3},y_{i_4}] =0$,
then considering proper polynomials, we have only sums of commutators, and by Jacobi identity we can reorder the variables of that commutator and consider only the case $[ y_{i_1},  y_{i_2},\dots,  y_{i_{m}} ]$ where $i_1 > i_2 < i_3 < \cdots < i_{m-2}$. 

\begin{proposition}
    The polynomials  $[y_{h_l},y_{h_1},y_{h_2}, \dots ,\widehat{y_{h_l}}, \dots,y_{h_m}]$ with $h_{l} > h_1 < h_2 < \cdots < h_m$ are linearly independent modulo $\Id_{\mathbb{Z}_2}(\mathcal{A}^3)$.
\end{proposition}
\begin{proof}
    Consider the polynomial given by \[ f = \sum_{l=2}^{m}\beta_l [y_{h_l},y_{h_1},y_{h_2}, \dots ,\widehat{y_{h_l}}, \dots,y_{h_m}]\] and suppose that $f\in \Id_{\mathbb{Z}_2}(\mathcal{A}^3)$. 
    Suppose there exists $l_0$ such that $\beta_{l_0} \neq 0$, and consider the evaluation $y_{h_{l_0}} = e_{2,3}$ and $y_{h_i} = (e_{1,1} + e_{3,3})$ for $i \neq l_0$. Then, $\beta_{l_0}e_{2,3} = 0$, therefore $\beta_{l_0} = 0$, a contradiction. Thus, we have the desired result. 
\end{proof}

\subsection{Case $n=1$}

We work with proper polynomials. 
Since $ [y_{i_1},y_{i_2}] [y_{i_3},y_{i_4}] = 0$ and we have only one $z_j$, then we only consider sums of commutators with products of two commutators. 

As $ [y_{i_1},y_{i_2}]z_{j_1}=0$,  when we have products of two commutators then $z_{j_1}$ is in the first commutator.

From the identity (c) $ [y_{i_1},y_{i_2},z_{j_1}] = - z_{j_1}[y_{i_1},y_{i_2}]$  and using the identity (d) we have $  [y_{i_1},y_{i_2},z_{j_1},y_{i_3}] = 0$. Therefore, if $z_{j_1}$ lies in a commutator $\omega$, that commutator has the form $\omega =  [z_{j_1},y_{i_1},y_{i_2}, \dots,y_{i_k}] $ or $\omega =  [y_{i_1},y_{i_2}, \dots,y_{i_k},z_{j_1}] = -z_{j_1} [y_{i_1},y_{i_2}, \dots,y_{i_k}]$.

Let us consider the case $\omega =  [z_{j_1},y_{i_1},y_{i_2}, \dots,y_{i_k}] $. Note that by Jacobi identity $ [z_{j_1},y_{i_1},y_{i_2}] + [y_{i_1},y_{i_2},z_{j_1}] + [y_{i_2},z_{j_1},y_{i_1}] = 0 $, so  $ [z_{j_1},y_{i_1},y_{i_2}] - z_{j_1}[y_{i_1},y_{i_2}] - [z_{j_1},y_{i_2},y_{i_1}] = 0 $ and
$ [z_{j_1},y_{i_2},y_{i_1}] = [z_{j_1},y_{i_1},y_{i_2}] - z_{j_1}[y_{i_1},y_{i_2}]$. Therefore, we can write the variables $y_i$ in any order in $\omega$.

If  $\omega =  [y_{i_1},y_{i_2}, \dots,y_{i_k},z_{j_1}] = -z_{j_1} [y_{i_1},y_{i_2}, \dots,y_{i_k}]$, as in the case $P_{m,0}$, we can reorder the variables $y_{i_1},y_{i_2}, \dots,y_{i_k}$ such that  $z_{j_1}[y_{i_1},y_{i_2}, \dots,y_{i_k}]$  is a linear combination of polynomials $z_{j_1}[y_{i_l},y_{i_1},y_{i_2}, \dots ,\widehat{y_{i_l}}, \dots,y_{i_s}]$  where  $i_{l} > i_1 < i_2 < \cdots < i_k$.

A nonzero product of commutators has the form $ [z_{j_1},y_{i_1},\dots,y_{i_s}] [y_{h_1},y_{h_2},\dots,y_{h_t}]$ and by Proposition \ref{pr:A_3} we reduce it to the form   $ z_{j_1}[y_{i_1},y_{i_2},\dots,y_{i_{s+t}}]$. 
Then, if $\mathcal{B}$ is the space of proper polynomials in the variables $z_{j_1}$, $y_{i_1}$, $y_{i_2}$, \dots, $y_{i_k}$, modulo $\Id_{\mathbb{Z}_2}(\mathcal{A}^3)$, $\mathcal{B}$ is spanned by \begin{itemize}
    \item $ [z_{j_1},y_{i_1},y_{i_2}, \dots,y_{i_k}]$, $i_1 < i_2 < \cdots < i_k$,
    \item $ z_{j_1}[y_{i_l},y_{i_1},y_{i_2}, \dots ,\widehat{y_{i_l}}, \dots,y_{i_k}]$, $i_{l} > i_1 < i_2 < \cdots < i_k$.
\end{itemize}

\begin{theorem}\label{th:A_3}
The following identities  form a basis for the $\mathbb{Z}_2$-graded identities for the algebra $\mathcal{A}^3$: 
\begin{itemize}
    \item[$(i)$] $ z_{i_1}z_{i_2}$,
    \item[$(ii)$] $ [y_{i_1},y_{i_2}][y_{i_3},y_{i_4}] $,
    \item[$(iii)$]  $ [y_{i_1},y_{i_2}]z_{i_1} $,
    \item[$(iv)$]  $ [z_{j_1}[y_{i_1},y_{i_2}],y_{i_3}]$.
\end{itemize}
  \end{theorem}

\begin{proof}
   Let us show that the polynomials $ [z_{j_1},y_{i_1},y_{i_2}, \dots,y_{i_m}]$ with  $i_1 < i_2 < \cdots < i_m$, and $ z_{j_1}[y_{h_l},y_{h_1},y_{h_2}, \dots ,\widehat{y_{h_l}}, \dots,y_{h_m}] $ with $h_{l} > h_1 < h_2 < \cdots < h_m$ are linearly independent modulo $\Id_{\mathbb{Z}_2}(\mathcal{A}^3)$. Consider $f$ the polynomial given by \[ f = \alpha[z_{j_1},y_{i_1},y_{i_2}, \dots,y_{i_m}] + \sum_{l=2}^{m}\beta_l z_{j_1}[y_{h_l},y_{h_1},y_{h_2}, \dots ,\widehat{y_{h_l}}, \dots,y_{h_m}] \] and suppose that $f\in \Id_{\mathbb{Z}_2}(\mathcal{A}^3).$ Now, considering the evaluation $z_{j_1} = e_{1,2}$ and $y_{i} = (e_{1,1} + e_{3,3})$ for $i = 1,\dots,m$ one has $\alpha e_{1,2} = 0$, so $\alpha = 0$. Then, \[ f = \sum_{l=2}^{m} \beta_l z_{j_1}[y_{h_l},y_{h_1},y_{h_2}, \dots ,\widehat{y_{h_l}}, \dots,y_{h_m}]=0.\] 
       Suppose there exists $l_0$ such that $\beta_{l_0} \neq 0$, and consider the evaluation $y_{h_{l_0}} = e_{2,3}$, $y_{h_i} = (e_{1,1}+e_{3,3})$ for $i \neq l_0$ and $z_{j_1} = e_{1,2}$. Then, $\beta_{l_0}e_{1,3} = 0$, therefore $\beta_{l_0} = 0$, a contradiction. Thus, we have the desired result. 
\end{proof}

\section{$S_m \times S_n$-characters of $\mathcal{A}^1$}
In this section, we consider the cocharacters of the $\mathbb{Z}_2$-graded algebra $\mathcal{A}^1$ (i.e., the algebra $A$ with the grading given by (\ref{eq:gra1}). We recall that a basis for the graded identities of $\mathcal{A}^1$ is given by the polynomials $[y_1,y_2]$ and $z_1z_2z_3$. 

Let $S_m$ act on the variables $y_1$, \dots, $y_m$ and let $S_n$ act on $z_1$, \dots, $z_n$, and consider the  cocharacter 
\[ \chi_{m,n}(\mathcal{A}^1) = \sum_{(\lambda, \mu) \vdash (m,n) } m_{\lambda,\mu}\chi_{\lambda} \otimes \chi_{\mu}.
 \]

Let $\lambda \vdash m$, $\mu \vdash n$,  and let $W_{\lambda,\mu}$  be a left irreducible $S_m \times S_n$-module.  If $T_{\lambda}$ is a tableau of shape  $\lambda$ and $T_{\mu}$  a tableau of shape $\mu$, then $W_{\lambda,\mu} \cong  F(S_m \times S_n)e_{T_\lambda}e_{T_\mu}$ with $S_m$ and $S_n$  acting on disjoint sets of integers.

From the identity $[y_1, y_2] = 0$ we have  $m_{(m),\emptyset} = 1$ and $m_{\lambda,\emptyset} = 0$ for $\lambda \neq (m)$. Also, from  $[y_1, y_2] = 0$, we obtain $m_{\lambda,\mu} = 0$ for $h(\lambda) >2$.  

Consider the case $n=1$. Since $P_{m,1}(\mathcal{A}^1) = \Span\{ y_{i_1}y_{i_2} \cdots y_{i_s} z  y_{j_1}y_{j_2} \cdots y_{i_t}  \}$, we consider the case $\lambda \vdash m$ with $\lambda = (p+q,p)$.

For every $i=0$, \dots, $q$ define the following two tableaux: $T_\lambda^{(i)}$ is the tableau
$$\begin{tabular}{|c|c|c|c|c|c|c|c|c|c|c|}
\hline$i+1$ & $i+2$ & $\cdots$ & $i+p$ & 1 & 2 & $\cdots$ & $i$ & $i+2 p+2$ & $\cdots$ & $2p+q+1$ \\
\hline$i+p+2$ & $i+p+3$ & $\cdots$ & $i+2 p+1$ & \multicolumn{7}{|c}{}  \\
\cline { 1 - 4 } 
\end{tabular},
$$
$$
T_\mu^{(i)}= \begin{tabular}{|l|}
\hline $i+p+1$  \\
\hline 
\end{tabular}.
$$
The associated polynomial is
 \begin{equation}\label{eq:yzy}
     a_{p,q}^{(i)}(y_1,y_2,z) = y_1^i \underbrace{\bar{y}_1 \cdots \tilde{y}_1}_p z \underbrace{\bar{y}_2 \cdots \tilde{y}_2}_p y_1^{q-i}, 
 \end{equation}
 where $\bar{}$ and $\tilde{}$ mean alternation on the corresponding elements. That is, 
 \[ f(\overline{x}_1,\dots,\overline{x}_i, x_{i+1}, \dots, x_n) = \sum_{\sigma\in S_i} (-1)^{\sigma} f(x_{\sigma(1)},\dots, x_{\sigma(i)}, x_{i+1} , \dots , x_n).\]
\begin{proposition} \label{pro:linear_inde}
    Modulo $\Id_{\mathbb{Z}_2}(\mathcal{A}^1)$, the polynomials $a_{p,q}^{(i)}(y_1,y_2,z)$ as in (\ref{eq:yzy}) are linearly independent, $i = 0$, \dots, $q$.  
\end{proposition}

\begin{proof}The proof is similar to that of Theorem 3 in \cite{valenti2002graded}. 
    Suppose that the polynomials $a_{p,q}^{(i)}$ are linearly dependent. Modulo  $\Id_{\mathbb{Z}_2}(\mathcal{A}^1)$, there exist $\alpha_i$, $i=0$, \dots, $q$ such that 
\[ \sum_{i=0}^{q}\alpha_i a_{p,q}^{(i)}(y_1,y_2,z) = 0. \]
    Let $t = \max\{ i \mid \alpha_i \neq 0 \}$. Then 
\begin{equation} \label{eq:alpha_t}
         \alpha_t a_{p,q}^{(t)}(y_1,y_2,z) + \sum_{i < t}\alpha_i a_{p,q}^{(i)}(y_1,y_2,z) = 0. 
    \end{equation}
    Considering $y_1 = y_1 + y_3$ in (\ref{eq:alpha_t}), we have that 

 $   \begin{aligned}
    & \alpha_t\left(y_1+y_3\right)^t \overline{\left(y_1+y_3\right)} \cdots\widetilde{\left(y_1+y_3\right)} z \bar{y}_2 \cdots \tilde{y}_2\left(y_1+y_3\right)^{q-t} \\
    & \quad +\sum_{i<t} \alpha_i\left(y_1+y_3\right)^i \overline{\left(y_1+y_3\right)} \cdots\widetilde{\left(y_1+y_3\right)}  z \bar{y}_2 \cdots \tilde{y}_2\left(y_1+y_3\right)^{q-i} = 0.
    \end{aligned}$

Let us consider the homogeneous component of degree $t + p$ in the variable $y_1$ and  degree $q - t$ in the variable $y_3$, 
\[ \alpha_t y_1^{t}\bar{y}_1\cdots\tilde{y}_1 z \bar{y}_2\cdots\tilde{y}_2 y_3^{q-t} + \cdots = 0.  \]
Substituting $y_1 = e_{1,1} + e_{3,3}$, $y_2 = y_3 = e_{2,2}$ and $z = e_{1,2}$, we obtain $\alpha_t\cdot e_{1,2} = 0$, so $\alpha_t = 0$ and we have the desired result.
\end{proof}

\begin{proposition} \label{pro:m_(1)}
     $m_{\lambda,(1)} = q + 1$, if $\lambda = (p+q,p) \vdash m$.
\end{proposition}
\begin{proof}
For every $i$,  $e_{T_{\lambda}^{(i)}}e_{T_{\mu}^{(i)}}(y_1,\dots,y_{m},z)$  is the complete linearization of the polynomial $a_{p,q}^{(i)}(y_1,y_2,z)$. By Proposition \ref{pro:linear_inde} $m_{\lambda,(1)} \geq q+1$ if $\lambda = (p+q,p) \vdash m$. Also,  given $T_{\lambda}$ and $T_{\mu}$ two tableaux and $f = e_{T_{\lambda}^{(i)}}e_{T_{\mu}^{(i)}}(y_1,\dots,y_{m},z)  \notin \Id_{\mathbb{Z}_2}(\mathcal{A}^1)$, any two alternating variables in $f$ must lie on different sides of $z$. Since $f$ is a linear combination (modulo $\Id_{\mathbb{Z}_2}(\mathcal{A}^1)$) of polynomials, each alternating on $p$ pairs of $y_i$'s, we have that $f$ is a linear combination (modulo $\Id_{\mathbb{Z}_2}(\mathcal{A}^1)$) of the polynomials $e_{T_{\lambda}^{(i)}}e_{T_{\mu}^{(i)}}$.  Hence  $m_{\lambda,(1)} = q + 1$, if $\lambda = (p+q,p) \vdash m$. 
\end{proof}

Now,  consider the case $n=2$,  $P_{m,2}(\mathcal{A}^1) = \Span\{ y_{i_1}y_{i_2} \cdots y_{i_s} z_{l_1}  y_{j_1}y_{j_2} \cdots y_{i_t} z_{l_2}  \}$.
For $\mu \vdash 2$ we have the possibilities $\mu = (2)$ and $\mu = (1,1)$.  Define the tableaux $T_{\lambda}^{(i)}$ as 
$$\begin{tabular}{|c|c|c|c|c|c|c|c|c|c|c|}
\hline$i+1$ & $i+2$ & $\cdots$ & $i+p$ & 1 & 2 & $\cdots$ & $i$ & $i+2 p+2$ & $\cdots$ & $2p+q+1$ \\
\hline$i+p+2$ & $i+p+3$ & $\cdots$ & $i+2 p+1$ & \multicolumn{7}{|c}{}  \\
\cline { 1 - 4 } 
\end{tabular}
$$
and
$$
T_{(2)}^{(i)}= \begin{tabular}{|l|l|}
\hline $i+p+1$ & $2p +q+ 2$ \\
\hline 
\end{tabular}, \quad 
T_{(1,1)}^{(i)}= \begin{tabular}{|l|}
\hline $i+p+1$ \\
\hline $2p +q+ 2$ \\
\hline 
\end{tabular} 
$$
Then, the associated polynomials are 
\begin{equation} \label{eq:yzyz}
    a_{p,q}^{(i)}(y_1,y_2,z) = y_1^i \underbrace{\bar{y}_1 \cdots \tilde{y}_1}_p z \underbrace{\bar{y}_2 \cdots \tilde{y}_2}_p y_1^{q-i}z, 
\end{equation}
and  respectively 
\begin{equation}\label{eq:yzyz2}
    a_{p,q}^{(i)}(y_1,y_2,z_1,z_2) = y_1^i \underbrace{\bar{y}_1 \cdots \tilde{y}_1}_p \check{z}_1 \underbrace{\bar{y}_2 \cdots \tilde{y}_2}_p y_1^{q-i} \check{z}_2, 
\end{equation}
First we consider the case of the polynomials (\ref{eq:yzyz}).

\begin{proposition} \label{pro:linear_inde2}
    Modulo $\Id_{\mathbb{Z}_2}(\mathcal{A}^1)$, the polynomials $a_{p,q}^{(i)}(y_1,y_2,z)$ as in (\ref{eq:yzyz}) are linearly independent, $i = 0$, \dots, $q$.  
\end{proposition}

\begin{proof} 
    Suppose that the polynomials $a_{p,q}^{(i)}$ are linearly dependent. Then, modulo  $\Id_{\mathbb{Z}_2}(\mathcal{A}^1)$ there exist $\alpha_i$, $i=0$, \dots, $q$ such that 
\[ \sum_{i=0}^{q}\alpha_i a_{p,q}^{(i)}(y_1,y_2,z) = 0. \]
    Let $t = \max\{ i \mid \alpha_i \neq 0 \}$. Then, 
\begin{equation} \label{eq:alpha_t2}
         \alpha_t a_{p,q}^{(t)}(y_1,y_2,z) + \sum_{i < t}\alpha_i a_{p,q}^{(i)}(y_1,y_2,z) = 0. 
    \end{equation}
    Considering $y_1 = y_1 + y_3$ in (\ref{eq:alpha_t2}), we have that 

 $   \begin{aligned}
    & \alpha_t\left(y_1+y_3\right)^t \overline{\left(y_1+y_3\right)} \cdots\widetilde{\left(y_1+y_3\right)} z \bar{y}_2 \cdots \tilde{y}_2\left(y_1+y_3\right)^{q-t}z \\
    & \quad +\sum_{i<t} \alpha_i\left(y_1+y_3\right)^i \overline{\left(y_1+y_3\right)} \cdots\widetilde{\left(y_1+y_3\right)}  z \bar{y}_2 \cdots \tilde{y}_2\left(y_1+y_3\right)^{q-i}z = 0.
    \end{aligned}$

We consider the homogeneous component of degree $t + p$ in $y_1$ and of degree $q - t$ in $y_3$, 
\[ \alpha_t y_1^{t}\bar{y}_1\cdots\tilde{y}_1 z \bar{y}_2\cdots\tilde{y}_2 y_3^{q-t}z + \cdots = 0.  \]
Substituting $y_1 = e_{1,1} + e_{3,3}$, $y_2 = y_3 = e_{2,2}$ and $z = e_{1,2} + e_{2,3}$, we obtain $\alpha_t\cdot e_{1,3} = 0$, so $\alpha_t = 0$ and we have the result.
\end{proof}

From Proposition \ref{pro:linear_inde2} and with a similar argument to Proposition \ref{pro:m_(1)}, we obtain

\begin{proposition} \label{pro:m_(2)}
    $m_{\lambda,(2)} = q + 1$, if $\lambda = (p+q,p) \vdash m$.
\end{proposition}

Finally, we consider the case of the polynomials (\ref{eq:yzyz2}), 
\[   a_{p,q}^{(i)}(y_1,y_2,z_1,z_2) = y_1^i \underbrace{\bar{y}_1 \cdots \tilde{y}_1}_p \check{z}_1 \underbrace{\bar{y}_2 \cdots \tilde{y}_2}_p y_1^{q-i} \check{z}_2.   \]

\begin{proposition} \label{pro:linear_inde3}
    Modulo $\Id_{\mathbb{Z}_2}(\mathcal{A}^1)$, the polynomials $a_{p,q}^{(i)}(y_1,y_2,z_1,z_2)$ as in (\ref{eq:yzyz2}) are linearly independent, $i = 0$,  \dots, $q$.  
\end{proposition}

\begin{proof} 
If we consider the evaluation  $z_1 = e_{1,2}$ and $z_2 = e_{2,3}$, then
    \[   a_{p,q}^{(i)}(y_1,y_2,z_1,z_2) = y_1^i \underbrace{\bar{y}_1 \cdots \tilde{y}_1}_p z_1 \underbrace{\bar{y}_2 \cdots \tilde{y}_2}_p y_1^{q-i} z_2.   \] 
    Suppose  the polynomials $a_{p,q}^{(i)}$ are linearly dependent. Modulo  $\Id_{\mathbb{Z}_2}(\mathcal{A}^1)$ there exist $\alpha_i$, $i=0$, \dots, $q$ such that 
\[ \sum_{i=0}^{q}\alpha_i a_{p,q}^{(i)}(y_1,y_2,z_1,z_2) = 0. \]
    Let $t = \max\{ i \mid \alpha_i \neq 0 \}$. Then, 
\begin{equation} \label{eq:alpha_t3}
         \alpha_t a_{p,q}^{(t)}(y_1,y_2,z_1,z_2) + \sum_{i < t}\alpha_i a_{p,q}^{(i)}(y_1,y_2,z_1,z_2) = 0. 
    \end{equation}
    Considering $y_1 = y_1 + y_3$ in (\ref{eq:alpha_t3}), we have that 

 $   \begin{aligned}
    & \alpha_t\left(y_1+y_3\right)^t \overline{\left(y_1+y_3\right)} \cdots\widetilde{\left(y_1+y_3\right)} \check{z}_1  \bar{y}_2 \cdots \tilde{y}_2\left(y_1+y_3\right)^{q-t}\check{z}_2 \\
    & \quad +\sum_{i<t} \alpha_i\left(y_1+y_3\right)^i \overline{\left(y_1+y_3\right)} \cdots\widetilde{\left(y_1+y_3\right)}  \check{z}_1  \bar{y}_2 \cdots \tilde{y}_2\left(y_1+y_3\right)^{q-i}\check{z}_2 = 0.
    \end{aligned}$

The homogeneous component of degree $t + p$ in $y_1$ and of degree $q - t$ in $y_3$ is
\[ \alpha_t y_1^{t}\bar{y}_1\cdots\tilde{y}_1 \check{z}_1  \bar{y}_2\cdots\tilde{y}_2 y_3^{q-t} \check{z}_2  + \cdots = 0.  \]
Substituting $y_1 = e_{1,1} + e_{3,3}$, $y_2 = y_3 = e_{2,2}$,  $z_1 = e_{1,2}$ and $ z_2 =  e_{2,3}$, we obtain $\alpha_t\cdot e_{1,3} = 0$, so $\alpha_t = 0$ and we are done.
\end{proof}

\begin{proposition} \label{pro:m_(1,1)}
    $m_{\lambda,\mu} = q + 1$, if $\lambda = (p+q,p) \vdash m$ and $\mu = (1,1)$.
\end{proposition}

\begin{proof}
    For every $i$,  $e_{T_{\lambda}^{(i)}}e_{T_{\mu}^{(i)}}(y_1,\dots,y_{m},z_1,z_2)$  is the complete linearisation of the element $a_{p,q}^{(i)}(y_1,y_2,z_1,z_2)$. By Proposition \ref{pro:linear_inde3}, $m_{\lambda,\mu} \geq q+1$ if $\lambda = (p+q,p) \vdash m$ and $\mu = (1,1)$. Also,  given $T_{\lambda}$ and $T_{\mu}$, and $f = e_{T_{\lambda}^{(i)}}e_{T_{\mu}^{(i)}}(y_1,\dots,y_{m},z_1,z_2)  \notin \Id_{\mathbb{Z}_2}(\mathcal{A}^1)$, any two alternating variables in $f$ must lie on different sides of $z_1$ or $z_2$, or the  two alternating variables are $z_1$ and $z_2$. Since $f$ is a linear combination (modulo $\Id_{\mathbb{Z}_2}(\mathcal{A}^1)$) of polynomials, each alternating on $p$ pairs of $y_i$'s, and alternating on $z_i$'s, we have that $f$ is a linear combination (modulo $\Id_{\mathbb{Z}_2}(\mathcal{A}^1)$) of the polynomials $e_{T_{\lambda}^{(i)}}e_{T_{\mu}^{(i)}}$.  Hence  $m_{\lambda,(1,1)} = q + 1$, if $\lambda = (p+q,p) \vdash m$. 
\end{proof}

Finally, based on the previous results, we have the following theorem on the graded cocharacters of the algebra $\mathcal{A}^1$.

\begin{theorem}
Let 
\[ \chi_{m,n}(\mathcal{A}^1) = \sum_{(\lambda, \mu) \vdash (m,n) } m_{\lambda,\mu}\chi_{\lambda} \otimes \chi_{\mu} \]
be the $(m,n)$-cocharacter of $\mathcal{A}^1$. Then
\begin{itemize}
    \item[(1)] $m_{\lambda,\emptyset} = 1$, if $\lambda = (m) \vdash m$;

    \item[(2)] $m_{\lambda,(1)} = q+1$, if $\lambda = (p+q,p) \vdash m$;

    \item[(3)] $m_{\lambda,\mu} = q+1$, if $\lambda = (p+q,p) \vdash m$ and $\mu \vdash 2$.    
\end{itemize}
    In all remaining cases $m_{\lambda,\mu} = 0$.
\end{theorem}

\section{Identities with involutions on $A$} \label{InvotionOnA}

We are interested in studying the identities of the algebra $A$, considering it as an algebra with involution. We can define on the algebra $A$ the involution $*$ obtained by reflecting a matrix along its secondary diagonal, that is  \[ \begin{pmatrix} d & a & c \\ 0 & g & b \\ 0 & 0 & d  \end{pmatrix}^{*} = \begin{pmatrix} d & b & c \\ 0 & g & a \\ 0 & 0 & d  \end{pmatrix}. \]
Then, 
\[ A^{+} = \begin{pmatrix} d & a & c \\ 0 & g & a \\ 0 & 0 & d  \end{pmatrix} \text{ and }  A^{-} =  \begin{pmatrix} 0 & a & 0 \\ 0 & 0 & -a \\ 0 & 0 & 0  \end{pmatrix}.
\]
This involution was studied in \cite{MISHCHENKO200066}.

Denote the free algebra with involution by $K\langle X, * \rangle = K\langle Y \cup  Z \rangle$ generated by symmetric and skew elements, that is \[ K\langle Y \cup  Z \rangle = K \langle y_{1}, z_{1},  y_{2},  z_{2}, , \dots \rangle \]
where $y_{i}$ stands for a symmetric variable  and $z_{i}$ for a skew-symmetric variable.

If in a polynomial it is allowed some variable to be either $y_i$ or $ z_i$, we denote it as  $x_i$, and  we set $|x_i|=0$ if $x_i=z_i$ and $|x_i|=1$ when $x_i=y_i$. Put  $|x_i x_j| = 0$  when the commutator $[x_i, x_j]$ is skew-symmetric, and $|x_i x_j|= 1$ if the commutator $[x_i, x_j]$ is symmetric. 

The following theorem describes a basis of $\Id^{*}(A)$.
\begin{theorem}
    The $T^{*}$-ideal $\Id^{*}(A)$ is generated by the following polynomials:
    
\begin{tabular}{llll} \label{th:Base_*}
    1. $z_1y_1z_2y_2z_3$ & 2. $[z_1,z_2]$ & 3. $[z_{1}y_1z_{2},y_2]$    &   4. $z_1y_1z_2 - z_2y_1z_1$ \\
    \multicolumn{4}{l}{ 5. $(-1)^{|x_1x_2|}[x_1,x_2][x_3,x_4] - (-1)^{|x_3x_4|}[x_3,x_4][x_1,x_2]$} \\
     \multicolumn{4}{l}{ 6. $(-1)^{|x_1x_2|}[x_1,x_2][x_3,x_4] - (-1)^{|x_1x_3|}[x_1,x_3][x_2,x_4] + (-1)^{|x_1x_4|}[x_1,x_4][x_2,x_3]$} \\
      \multicolumn{4}{l}{ 7. $ [x_1,x_2]z_5[x_3,x_4]$} 
\end{tabular}
\end{theorem}
The proof of Theorem \ref{th:Base_*} will be a consequence of the following propositions.

Let $\mathcal{I}$ be the $T^{*}$-ideal generated by the identities (1) to (7) of Theorem \ref{th:Base_*}. 

\begin{proposition} \label{pr:order*_m=2}
    The following is an identity modulo $\mathcal{I}$
    \begin{equation}
        y_{\sigma(i_1)}y_{\sigma(i_2)}\cdots y_{\sigma(i_s)}z_1y_{\rho(j_1)}y_{\rho(j_2)}\cdots y_{\rho(j_t)}z_2 = y_{i_1}y_{i_2}\cdots y_{i_s}z_1y_{j_1}y_{j_2}\cdots y_{j_t}z_2,
    \end{equation}
    for $\sigma \in S_s$ and $\rho \in S_t$.
\end{proposition}
\begin{proof}
    Consider the case $y_{i_2}y_{i_1}z_1y_{j_2}y_{j_1}z_2$ and note that $[y_1,y_2] \in A^{-}$. Then, 
    \begin{align*}
        y_{i_2}y_{i_1}z_1y_{j_2}y_{j_1}z_2 & = y_{i_1}y_{i_2}z_1y_{j_2}y_{j_1}z_2 + [y_{i_2},y_{i_1}]z_1y_{j_2}y_{j_1}z_2  = y_{i_1}y_{i_2}z_1y_{j_2}y_{j_1}z_2 \\
        & =  y_{i_1}y_{i_2}z_1y_{j_1}y_{j_2}z_2 + y_{i_1}y_{i_2}z_1[y_{j_2},y_{j_1}]z_2  = y_{i_1}y_{i_2}z_1y_{j_1}y_{j_2}z_2.
    \end{align*}
It follows that  \[ 
        y_{\sigma(i_1)}y_{\sigma(i_2)}\cdots y_{\sigma(i_s)}z_1y_{\rho(j_1)}y_{\rho(j_2)}\cdots y_{\rho(j_t)}z_2 = y_{i_1}y_{i_2}\cdots y_{i_s}z_1y_{j_1}y_{j_2}\cdots y_{j_t}z_2,\] 
       for $\sigma \in S_s$ and $\rho \in S_t$.
\end{proof}

Let $P_{n,m}$ be the set of multilinear polynomials in $n$ symmetric variables and $m$ skew-symmetric variables. By Id. 1., $P_{n,m} = 0$ if $m>2$. Then, we consider the cases $m = 0$, $1$, $2$ separately.

\subsection{Case $m = 0$}

Let $\Gamma_{n,0}(\mathcal{I})$ be the subspace of the $Y$-proper multilinear polynomials in the variables $y_1$, \dots, $y_m$  of the relatively free algebra $F\langle Y,Z, * \rangle/\mathcal{I}$.

Since $[x_1, x_2][x_3, x_4][x_5, x_6]$ belongs to $\mathcal{I}$, the vector
space $\Gamma_{n,0}(\mathcal{I})$ is spanned  by the proper polynomials \[ [x_{j_1}, \dots , x_{j_s}] [x_{k_1}, \dots , x_{k_t}], \]
where $s$, $t \geq 0$, $s \neq 1$, $t \neq 1$,   $j_1 > j_2 < \cdots < j_s$, and $k_1 > k_2 < \cdots < k_t$.

Note that $A^{+} = UT_3(K)^{+}$. We use a result of \cite{di2006involutions}.

\begin{definition}
    A polynomial $f$ is called {\sl $S_3$-standard} if $f$ is either of type $[y_{j_1}, \dots , y_{j_n}]$ or $[y_{j_1}, \dots , y_{j_{n-2}}][y_{k_1}, y_{k_2}]$, where the commutator $[y_{i_1}, \dots , y_{i_s}]$ satisfies $i_1 > i_2 < \cdots < i_s$, and  if $f$ is of the second type we have that $j_1 >k_1$, ${j_2} >k_2$.
\end{definition}
The next result is a particular case of Proposition 5.8 and Lemma 6.2 of \cite{di2006involutions}.

\begin{proposition} \label{pro:base*_m=0}
The following statements hold:
    \begin{itemize}
   \item[$(i)$] The $S_3$-standard polynomials span the vector space $\Gamma_{n,0}(A)$.
        
        \item[$(ii)$] The $S_3$-standard polynomials in $\Gamma_{n,0}$ are linearly independent modulo the ideal $\Id^{*}(A)$.
    \end{itemize}
\end{proposition}

\subsection{Case $m = 1$}
First, we describe a spanning set of  $ P_{n,1}(A)$.

\begin{proposition}\label{pr:spn*_m=1}
     $ P_{n,1}(A)$ is spanned by  elements of the form
\begin{itemize}
    \item[$(i)$] $ y_{i_1}y_{i_2}\cdots y_{i_s}zy_{i_{s+1}}y_{i_{s+2}}\cdots y_{i_n}$, with $i_1<\cdots<i_s$ and $i_{s+1} < \cdots < i_n$;

    \item[$(ii)$]  $ y_{i_1}y_{i_2}\cdots y_{i_s}zy_{i_{s+1}}y_{i_{s+2}}\cdots y_{i_{n-2}}[y_{i_{n-1}},y_{i_{n}}]$, with $i_1<\cdots<i_s$, $i_{s+1} < \cdots < i_{n-1}$ and $i_{n}<i_{n-1}$.
\end{itemize}
\end{proposition}

\begin{proof}
Let $ \omega \in P_{n,1}(A)$ be a monomial. Then, $\omega =  y_{i_1}y_{i_2}\cdots y_{i_s}zy_{j_1}y_{j_2}\cdots y_{j_t}$ with $s+t = n$. 

Suppose $\omega = \omega_1 z \omega_2$, there are two cases for $\omega_2$:

    Case 1. {\sl $\omega_2$ with indices in non-increasing order}. 
    
    Consider $\omega = \omega_1 z \omega_2$ with $\omega_2 = w_{2,1}y_{l_2}y_{l_1}w_{2,2}$ and $y_{l_2},y_{l_1}$ the first pair of variables such that $l_2 > l_1$. 

     We can rewrite $\omega_2$ as $ \omega_2 = w_{2,1}y_{l_1}y_{l_2}w_{2,2} +  w_{2,1}[y_{l_2},y_{l_1}]w_{2,2}$, then 
   \[
     \omega =  \omega_1 z  w_{2,1}y_{l_1}y_{l_2}w_{2,2} + \omega_1 z  w_{2,1}[y_{l_2},y_{l_1}]w_{2,2} 
= \omega_1 z \overline{\omega}_2 +  \omega_1 w_{2,2} z  w_{2,1}[y_{l_2},y_{l_1}].
\]
Note that by Proposition \ref{pr:order*_m=2}, in the element $\omega_1 w_{2,2} z  w_{2,1}[y_{l_2},y_{l_1}] $ we can reorder the variables of  $\omega_1 w_{2,2}$ and $ w_{2,1}$ (separately) with the condition that if $ w_{2,1} = y_{k_1}\dots y_{k_p}$, then $y_{k_1} < \cdots < y_{k_p} < y_{l_2}$ and $  y_{l_1} < y_{l_2}$. 

    We repeat the process for $\omega_1z\overline{\omega}_2$ until getting an ascending order in the indices of the $y$'s to the right side of $z$ in the element without commutator. 
    Thus, $\omega = \omega_1 z \omega_2$ can be written as \[ \omega = \omega_1 z  y_{h_1}\cdots y_{h_r} + \sum_{I,J} \alpha_{I,J} y_{i_1}\cdots y_{i_s} z y_{j_1}\cdots y_{j_{t-2}}[y_{j_{t-1}},y_{j_t}] \] with ${h_1}<\cdots <{h_r}$,  ${i_1}<\cdots < {i_s}$, ${j_1}<\cdots< y_{j_{t-1}}$, and $j_t < {j_{t-1}}$. 

If the indices of the variables of $\omega_1$ are in increasing order, then we have $\omega$ in the desired form.

Case 2. {\sl $\omega_2$ with indices in increasing order}. 

     Let $\omega = \omega_1 z \omega_2$ with $\omega_1 = w_{1,1}y_{l_2}y_{l_1}w_{1,2}$, and let $y_{l_2},y_{l_1}$ be the first pair of variables such that $l_2 > l_1$ and the indices of the variables of $\omega_2$ increase.

     We  rewrite $\omega_1$ as $ \omega_1 = w_{1,1}y_{l_1}y_{l_2}w_{1,2} +  w_{1,1}[y_{l_2},y_{l_1}]w_{1,2}$, then 
    \[
     \omega  =  w_{1,1}y_{l_1}y_{l_2}w_{1,2} z \omega_2   +  w_{1,1}[y_{l_2},y_{l_1}]w_{1,2} z \omega_2   = \overline{\omega}_1z\omega_2 +  w_{1,1}\omega_2 z w_{1,2} [y_{l_2},y_{l_1}].
    \]
We apply once again Proposition \ref{pr:order*_m=2} to the element $ w_{1,1}\omega_2 z w_{1,2} [y_{l_2},y_{l_1}] $. Thus we reorder the variables of  $w_{1,1}\omega_2$ and $  w_{1,2}$ (separately). Then, \[ \omega =  \overline{\omega}_1z\omega_2 + \varpi_1 z \varpi_2 [y_{l_2},y_{l_1}], \] with the indices of the variables of $\omega_2$, $ \varpi_1$ and $\varpi_2$ in increasing order, respectively. 

    Note that there is no  relation between the indices of the variables of $\varpi_2$ and those of the variables of $[y_{l_2},y_{l_1}]$. Take
    \[ \omega =  \overline{\omega}_1z\omega_2 + \varpi_1 z \varpi_2 y_{l_2}y_{l_1} -  \varpi_1 z \varpi_2 y_{l_1}y_{l_2}. \]
    For $ \varpi_1 z \varpi_2 y_{l_2}y_{l_1} $ and  $ \varpi_1 z \varpi_2 y_{l_1}y_{l_2}$ we have two options: 
    \begin{itemize}
        \item[$\bullet$] The indices of the variables of $\varpi_2 y_{l_2}y_{l_1}$ or $\varpi_2 y_{l_2}y_{l_1}$ are not in increasing order. In this case, we proceed as in Case 1 for $\varpi_1z\varpi_2 y_{l_2}y_{l_1}$ or $\varpi_1 z\varpi_2 y_{l_2}y_{l_1}$.

        \item[$\bullet$]    The indices of the variables of $\varpi_2 y_{l_2}y_{l_1}$ are in increasing order. In this case, we have the desired order for $ \varpi_1 z \varpi_2 y_{l_1}y_{l_2}$. 
    \end{itemize}
    Finally, if necessary, we repeat the process for $ \overline{\omega}_1 z \omega_2$. 

This process of rearranging the variables using cases 1. and 2., has finitely many steps, so eventually it ends.
\end{proof}

Now, we show that the polynomials of Proposition \ref{pr:spn*_m=1} are  linearly independent. 

\begin{proposition}\label{pr:base*_m=1}
     The polynomials of the form 
\begin{itemize}
    \item[$(I)$] $ y_{i_1}y_{i_2}\cdots y_{i_s}zy_{i_{s+1}}y_{i_{s+2}}\cdots y_{i_n}$, with $i_1<\cdots<i_s$ and $i_{s+1} < \cdots < i_n$;

    \item[$(II)$]  $ y_{i_1}y_{i_2}\cdots y_{i_s}zy_{i_{s+1}}y_{i_{s+2}}\cdots y_{i_{n-2}}[y_{i_{n-1}},y_{i_{n}}]$, with $i_1<\cdots<i_s$, $i_{s+1} < \cdots < i_{n-1}$ and $i_{n}<i_{n-1}$,
\end{itemize}
are linearly independent modulo $\Id^{*}(A)$.
\end{proposition}

\begin{proof}
    Let $f$ be a linear combination of elements of the form $(I)$ and $(II)$, such that $f$ is a $*$-polynomial identity of $A$. 

    Suppose that not  all elements of $f$ of the form $(I)$ have non-zero coefficients, and consider $m =  y_{l_1}\cdots y_{l_s}zy_{l_{s+1}}\cdots y_{l_n}$.  
    Perform the evaluation \begin{equation} \label{eq:ev1_m=1}
         y_{l_1} = \cdots = y_{l_s} = e_{1,1} + e_{3,3}, \quad z = e_{1,2} - e_{2,3}, \quad y_{l_{s+1}} = \cdots = y_{l_n} = e_{2,2}.
    \end{equation}
    Since $e_{1,1} + e_{3,3}$ and $e_{2,2}$ commute, the elements with commutator in $f$ vanish.
    
    Now, suppose there exists $\overline{m} =  y_{j_1}\cdots y_{j_t}zy_{j_{t+1}}\cdots y_{j_n} $ of the form $(I)$ such that $\overline{m}$ is non-zero under the evaluation (\ref{eq:ev1_m=1}). Then, in $\overline{m}$ the $e_{2,2}$'s can only be substituted on one side of $z$ and similarly for the $(e_{1,1} + e_{3,3})$'s. Then we have two possibilities: $\{ j_1, \dots, j_t \} = \{ l_1, \dots l_s\} $ and  $\{ j_{t+1}, \dots, j_n \} = \{ l_{s+1}, \dots l_n\}$, or $\{ j_1, \dots, j_t \} =  \{ l_{s+1}, \dots l_n\}   $, and  $\{ j_{t+1}, \dots, j_n \} = \{ l_1, \dots l_s\} $. In the first case $m = \overline{m}$ and  $m$ has zero coefficient in $f$.  In the second case the evaluations of $m$ and $\overline{m}$ are $e_{1,2}$ and $e_{2,3}$, respectively. Therefore, $m$ and $\overline{m}$ have coefficient zero.

Suppose  not all  elements with one commutator have zero coefficient, and consider \[ m =  y_{i_1}y_{i_2}\cdots y_{i_k}zy_{i_{k+1}}y_{i_{k+2}}\cdots y_{i_{n-2}}[y_{i_{n-1}},y_{i_{n}}]  \] with the following properties:
    \begin{itemize}
        \item[$(i)$] its coefficient is non-zero,
        \item[$(ii)$] the number $n-k-2$ of variables between $z$ and the commutator is the largest among all elements with non-zero coefficient,
        \item[$(iii)$] $i_{n-1}$ is the least of all elements such that the properties $(i)$ and $(ii)$ hold.  
    \end{itemize}
    Consider the evaluation \begin{equation} \label{eq:ev2_m=1}
         y_{i_1} = \cdots = y_{i_k} = E,  z = e_{1,2} - e_{2,3},  y_{i_{k+1}} = \cdots = y_{i_{n-1}} = e_{2,2},  y_{i_n} = e_{1,2}+e_{2,3}.
    \end{equation}
    Let $\overline{m} = y_{j_1}\cdots y_{j_s}zy_{j_{s+1}}\cdots y_{j_{n-2}}[y_{j_{n-1}},y_{j_{n}}]  $ be a polynomial of the form $(II)$, such that  $\overline{m}$ does not vanish under the substitution (\ref{eq:ev2_m=1}). We look for the conditions   $\overline{m}$ must satisfy.

    Since $E$ and $e_{2,2}$ commute, the elements to be substituted in the commutator will be $e_{1,2} + e_{2,3} $ and $e_{2,2}$. Also  $e_{2,2}$  can only be substituted on the right side of $z = (e_{1,2}-e_{2,3})$, and since $n-k-2 \geq n-s-2$, then the $E$'s only can be substituted on the left side of $z$.  Then $s = k$  and $\{ j_1, \dots, j_s \} = \{ i_1, \dots i_k\} $. 

If $j_{n-1} \notin \{i_{n-1},i_{n}\}$ then $j_{n-1} \in \{ i_{k+1}, \dots ,i_{n-2} \}$, then $j_{n-1}<i_{n-1}$, but this contradicts the hypothesis of minimality of $i_{n-1}$.   So  $j_{n-1} \in \{i_{n-1},i_{n}\}$, and since $i_n < i_{n-1}$ we have $j_{n-1} = i_{n-1}$ and $j_n = i_n$. We conclude  $\{ i_{k+1}, \dots, i_{n_2} \} = \{ j_{k+1}, \dots, j_{n_2} \} $, $m = \overline{m}$ and the coefficient of $m$ is zero.    
\end{proof}

\subsection{Case $m = 2$}
By Proposition \ref{pr:order*_m=2} and the identities $3.$ and $4.$, $P_{n,2}(A)$ is spanned by monomials of the form  \[ y_{i_1}y_{i_2}\cdots y_{i_s}z_1y_{j_1}y_{j_2}\cdots y_{j_t}z_2 \] with $i_1<i_2\cdots<i_s$, $j_1<j_2\cdots<j_t$ and $s+t = n$.

\begin{proposition} \label{pro:base*_m=2}
    The polynomials 
        \begin{equation} \label{eq:*_m=2}
        y_{i_1}y_{i_2}\cdots y_{i_s}z_1y_{j_1}y_{j_2}\cdots y_{j_t}z_2 \end{equation}
        with $i_1<i_2\cdots<i_s$, $j_1<j_2\cdots<j_t$ and $s+t = n$, are linearly independent modulo $\Id(A)$. 
\end{proposition}
\begin{proof}
    Suppose that \[ f = \sum_{I,J}\alpha_{I,J} y_{i_1}y_{i_2}\cdots y_{i_s}z_1y_{j_1}y_{j_2}\cdots y_{j_t}z_2\] is a $*$-identity of $A$. Choose  $\alpha_{I',J'} \neq 0$, such that $|J'|=t'$ is the largest possible. Now, consider the substitution $y_{i_1} = \cdots = y_{i_{s'}} = E$, $y_{j_1} = \cdots = y_{j_{t'}} = e_{2,2}$ and $z_1 = z_2 = (e_{1,2} - e_{2,3})$. Considering the order of the indices and that $|J'|$ is the largest, the only non-zero term of $f$ after the substitution will be the one with coefficient $\alpha_{I',J'}$. Hence $\alpha_{I',J'} = 0$, a contradiction.  Thus, the polynomials in Equation (\ref{eq:*_m=2}) are linearly independent. 
\end{proof}

Finally, by Proposition \ref{pro:base*_m=0}, Proposition \ref{pr:spn*_m=1}, Proposition \ref{pr:base*_m=1} and Proposition  \ref{pro:base*_m=2}  we have that Theorem \ref{th:Base_*} holds.

\section{Graded involution}

In this section we consider $A$ as a graded algebra with the grading as in (\ref{eq:gra1}), and the involution from Section \ref{InvotionOnA}. Then $*$ is a graded involution and \[ A_{0}^{+} = A_0, \qquad A_{0}^{-} = \{0\}, \qquad A_{1}^{+} = K(e_{1,2} + e_{2,3}), \qquad A_{1}^{-} = K(e_{1,2} - e_{2,3}).\]

\subsection{Graded $*$-polynomial identities of $A$}

Consider the free algebra with involution $K\langle Y \cup  Z, * \rangle$, generated by symmetric and skew elements of even and odd degree, that is \[ K\langle Y \cup  Z, * \rangle = K \langle y_{1}^{+}, y_{1}^{-}, z_{1}^{+}, z_{1}^{-}, y_{2}^{+}, y_{2}^{-}, z_{2}^{+}, z_{2}^{-}, \dots \rangle \]
where $y_{i}^{+}$ stands for a symmetric variable of even degree, $y_{i}^{-}$ for a skew variable of even degree, $z_{i}^{+}$ for a
symmetric variable of odd degree and $z_{i}^{-}$ for a skew variable of odd degree.

The following theorem gives us a basis of $\Id_{\mathbb{Z}_2}^{*}(A)$.

\begin{theorem} \label{th:Base_gr_*}
    The $T_2^{*}$-ideal $\Id_{\mathbb{Z}_2}^{*}(A)$ is generated, as a $T_2^{*}$-ideal, by the following polynomials:
    
\begin{tabular}{lll} 
    1. $y^{-}$ & 2. $[y_{1}^{+},y_{2}^{+}]$ & 3. $[z_{1}^{+},z_{2}^{+}]$       \\
    4. $[z_{1}^{-},z_{2}^{-}]$  &  5. $z_{1}^{+} \circ z_{1}^{-}$ & 6. $[y^{+},z_{1}z_{2}]$ \\
    7. $[y_{1}^{+},z_{1}y_{2}^{+}z_{2}]$ &  8.  $z_1 z_2 z_3$ & 9.  $ z_{1}^{+}y^{+}z_{2}^{+} - z_{2}^{+}y^{+}z_{1}^{+}$ \\
    10. $ z_{1}^{-}y^{+}z_{2}^{-} - z_{2}^{-}y^{+}z_{1}^{-}$ & 11.  $ z_{1}^{-}y^{+}z_{2}^{+} + z_{2}^{+}y^{+}z_{1}^{-}$ & 
\end{tabular}
\end{theorem}

\begin{proof}
   One easily sees that the polynomials 1. -- 11. are graded $*$-identities. Let $P_{n_1,n_2,n_3,n_4}(A)$ be the set of multilinear graded $*$-polynomials modulo $\Id_{\mathbb{Z}_2}^{*}(A)$ in $n_1$ symmetric variables of degree $0$, $n_2$ skew variables of degree $0$,  $n_3$ symmetric variables of degree $1$ and $n_1$ skew variables of degree $1$. From identities 1. and 8., we have that $P_{n_1,n_2,n_3,n_4}(A) = \{ 0 \}$ if $n_2>0$ or $n_3 + n_4 >2$. 

    \begin{enumerate} 
        \item[(i)] {\sl Case $n_3 + n_4 = 0$.} From identity 2., it is easy to see that \[  P_{n_1,0,0,0}(A) = \Span\{ y_1^{+}y_2^{+}\cdots y_{n_1}^{+} \}. \]
        \item[(ii)] {\sl Case $n_3 + n_4 = 1$.} We consider the case $n_3 = 1$ (the case $n_4 = 1$ is analogous). Again, by identity 2., we see that  \begin{equation} \label{eq:*_gr_z-1}
            P_{n_1,0,1,0}(A) = \Span\{ y_{i_1}^{+} \cdots y_{i_s}^{+}z^{+}y_{j_1}^{+}\cdots y_{j_t}^{+} \}, 
        \end{equation}  with $s+t = n_1$, $i_1<\dots<i_s$ and $j_1 < \cdots< j_t$.  

But the monomials $y_{i_1}^{+} \cdots y_{i_s}^{+}z^{+}y_{j_1}^{+}\cdots y_{j_t}^{+}$ are linearly independent. Indeed, let \[ f = \sum_{I,J}\alpha_{I,J}y_{i_1}^{+} \cdots y_{i_s}^{+}z^{+}y_{j_1}^{+}\cdots y_{j_t}^{+}.\]
        Suppose $f\equiv 0$ modulo $Id_{\mathbb{Z}_2}^{*}(A)$ and that there are   sets of indices $I_0$ and $J_0$, such that $\alpha_{I_0,J_0} \neq 0$. Considering the evaluation $y_i^{+} = e_{1,1} + e_{3,3}$, $y_{j}^{+} = e_{2,2}$, $z^{+} = e_{1,2} + e_{2,3}$, for $i \in I_0$, $j \in J_0$, we obtain $\alpha_{I_0,J_0}e_{1,2} = 0$. So, $\alpha_{I_0,J_0} = 0$, a contradiction. Then,  the polynomials in (\ref{eq:*_gr_z-1}) are linearly independent. 

         \item[(iii)] {\sl Case $n_3 + n_4 = 2$.} Consider the case $n_3 = 2$ (the cases $n_4 = 2 $ and $n_3 = n_4 = 1$  are analogous). From the identities 2., 3., 6., 7. and 9., we conclude  
         \begin{equation} \label{eq:*_gr_z-2}
            P_{n_1,0,2,0}(A) = \Span\{ y_{i_1}^{+} \cdots y_{i_s}^{+}z_{1}^{+}y_{j_1}^{+}\cdots y_{j_t}^{+}z_{2}^{+} \}, 
        \end{equation}  with $s+t = n_1$, $i_1<\dots<i_s$ and $j_1 < \dots j_t$. 

        In order to see that the polynomials in (\ref{eq:*_gr_z-2}) are linearly independent modulo $Id_{\mathbb{Z}_2}^{*}(A)$, suppose that \[ \sum_{I,J}\alpha_{I,J}y_{i_1}^{+} \cdots y_{i_s}^{+}z_{1}^{+}y_{j_1}^{+}\cdots y_{j_t}^{+}z_{2}^{+} = 0 \qquad \mod \Id_{\mathbb{Z}_2}^{*}(A)  \]
        and that there are  index sets $I_0$ and $J_0$, such that $\alpha_{I_0,J_0} \neq 0$. Considering the evaluation $y_i^{+} = e_{1,1} + e_{3,3}$, $y_{j}^{+} = e_{2,2}$, $z_{l}^{+} = e_{1,2} + e_{2,3}$, for $i \in I_0$, $j \in J_0$, we obtain $\alpha_{I_0,J_0}e_{2,3} = 0$. So, $\alpha_{I_0,J_0} = 0$, a contradiction. Then,  the polynomials in (\ref{eq:*_gr_z-2}) are linearly independent.         
    \end{enumerate}
    Therefore, from items (i), (ii) and (iii), we conclude that the set of polynomials 1.-11. determines a basis for the graded $*$-identities of the algebra $A$.
\end{proof}

\subsection{Cocharacters of $(A,gr,*)$}

Let $\langle \lambda \rangle = (\lambda(1), \lambda(2), \lambda(3), \lambda(4))$ be a multipartition of $(n_1,n_2,n_3,n_4)$, i.e. $\lambda(i)\vdash n_i$. We are interested in computing the $(n_1 , \dots, n_4 )$-th cocharacter of $(A,gr,*)$,

\begin{equation}\label{eq:cocharacter_A_gr_*}
    \chi_{n_1,n_2,n_3,n_4}(A) = \sum_{\langle \lambda \rangle \vdash (n_1,n_2,n_3,n_4)} m_{\langle \lambda \rangle} \chi_{\lambda(1)}\otimes \cdots  \otimes \chi_{\lambda(4)}.
\end{equation}
Since $\dim (\mathcal{A}^1)^{+} = \dim (\mathcal{A}^1)^{-} = 1$, then $m_{\langle \lambda \rangle} = 0$ if $h(\lambda(i))>1$ for $i = 3$, $4$. Also, by the  graded $*$-identities of $A$, we have  $m_{\langle \lambda \rangle} = 0$ if $h(\lambda(1))>2$.

In the following results, we consider only the case $h(\lambda(i))\leq 1$ for $i = 3$, $4$. 
First, we consider the case of  even degree variables only or  odd degree variables only.

\begin{proposition}
        If either $\langle \lambda \rangle = ((n_1), \emptyset, \emptyset, \emptyset)$ with $n_1 >0$ or $\langle \lambda \rangle = (\emptyset, \emptyset, (n_3), (n_4))$ with $ 0 < n_3 + n_4  \leq 2$, then $m_{\langle \lambda \rangle} = 1$ in (\ref{eq:cocharacter_A_gr_*}).
\end{proposition}
\begin{proof}
    Let $\langle \lambda_1 \rangle = ((n_1), \emptyset, \emptyset, \emptyset)$ and $\langle \lambda_2 \rangle = (\emptyset, \emptyset, (n_3), (n_4))$ as in the statement. Then $\omega_1 = (y_1^{+})^{n_1}$ and $\omega_2 = (z_{1}^{+})^{n_3}(z_{1}^{-})^{n_4}$ are highest weight vectors corresponding to the multipartitions $\langle \lambda_1 \rangle $ and $\langle \lambda_2 \rangle$ respectively. Since  $\omega_1$ and $\omega_2$ are not polynomials identities of $A$ then $m_{\langle \lambda_i \rangle} \geq 1$. By the identities in Theorem \ref{th:Base_gr_*} we conclude that $\omega_1$ and $\omega_2$ are the only (up to a scalar) highest weight vectors corresponding to $\langle \lambda_1 \rangle $ and $\langle \lambda_2 \rangle$. Therefore, $m_{\langle \lambda_i \rangle} = 1$ for $i=1$, $2$. 
\end{proof}

Before dealing with the case $n_3 + n_4 = 1$, we state a technical result (similar results can be found in \cite{cirrito2014ordinary,giambruno2019superalgebras}). 

\begin{proposition} \label{pr:equation}
    Modulo  $\Id_{\mathbb{Z}_2}^{*}(A)$, the following equality holds:
    \begin{equation} \label{eq:equation}
    \begin{split}
      &  \underbrace{\overline{y}_{1}^{+} \cdots \widetilde{y}_{1}^{+} }_p (y_{1}^{+})^{i_1-p} z^{+}  (y_{1}^{+})^{i_2-p} \underbrace{ \overline{y}_{2}^{+} \cdots \widetilde{y}_{2}^{+} }_p  \\       
        = & \sum_{j=0}^{p} (-1)^{j} \binom{p}{j} (y_{1}^{+})^{i_1-j} (y_{2}^{+})^{j} z^{+} (y_{1}^{+})^{i_2-p+j} (y_{2}^{+})^{p-j},
    \end{split}
     \end{equation}
    where $i_1,$ $i_2 \geq 0$ and $p\geq 1$.
\end{proposition}
\begin{proof}
    We prove it by induction on $p$. 
    The case $p=1$ is a straightforward computation. Let $p>1$ and let $\omega$ be the polynomial (\ref{eq:equation}). 
    \[
        \begin{split}
            \omega = & \underbrace{\overline{y}_{1}^{+} \cdots \widetilde{y}_{1}^{+} }_p (y_{1}^{+})^{i_1-p} z^{+}  (y_{1}^{+})^{i_2-p} \underbrace{ \overline{y}_{2}^{+} \cdots \widetilde{y}_{2}^{+} }_p  \\ 
             = & \underbrace{\overline{y}_{1}^{+} \cdots \widetilde{y}_{1}^{+} }_{p-1} (y_{1}^{+})^{i_1-(p-1)} z^{+}  (y_{1}^{+})^{i_2-1-(p-1)} {y}_{2}^{+} \underbrace{ \overline{y}_{2}^{+} \cdots \widetilde{y}_{2}^{+} }_{p-1}  \\
             & - \underbrace{\overline{y}_{1}^{+} \cdots \widetilde{y}_{1}^{+} }_{p-1} (y_{1}^{+})^{i_1-1-(p-1)} y_{2}^{+} z^{+}  (y_{1}^{+})^{i_2-(p-1)} \underbrace{ \overline{y}_{2}^{+} \cdots \widetilde{y}_{2}^{+} }_{p-1}.  \\
        \end{split}
      \]
    Applying induction to $p-1$ we obtain 
    \begin{equation} \label{eq:equa-2}
         \begin{split}
            \omega = & \sum_{j=0}^{p-1} (-1)^{j} \binom{p-1}{j} (y_{1}^{+})^{i_1-j} (y_{2}^{+})^{j} z^{+} (y_{1}^{+})^{(i_2-1) -(p-1)+j} (y_{2}^{+})^{(p-1)-j+1} \\ 
              & - \sum_{j=0}^{p-1} (-1)^{j} \binom{p-1}{j} (y_{1}^{+})^{(i_1-1)-j} (y_{2}^{+})^{j+1} z^{+} (y_{1}^{+})^{i_2-(p-1)+j} (y_{2}^{+})^{(p-1)-j}.
        \end{split}
    \end{equation}
The $j$-th monomial of the second summand is similar to the $j+1$-th monomial of the first summand, for every $j = 0$, \dots, $p-2$.
      Considering the sum of these monomials, for the corresponding coefficients we have 
      \[ 
            (-1)^{j+1}\binom{p-1}{j+1} - (-1)^{j}\binom{p-1}{j}  = (-1)^{j+1}\left( \binom{p-1}{j+1} + \binom{p-1}{j}   \right) 
             = (-1)^{j+1}\binom{p}{j+1}.
       \]
       The sum of the similar monomials of (\ref{eq:equa-2}) corresponds to the $j+1$-th monomial of
       \[ \omega_1 = \sum_{j=0}^{p} (-1)^{j} \binom{p}{j} (y_{1}^{+})^{i_1-j} (y_{2}^{+})^{j} z^{+} (y_{1}^{+})^{i_2-p+j} (y_{2}^{+})^{p-j},  \]
       for $j = 0$, \dots, $p-2$. But the first monomial of the first summand of (\ref{eq:equa-2}) is equal to the first monomial of $\omega_1$, whereas the $(p-1)$-th monomial of the second summand of (\ref{eq:equa-2})  is equal to the $p$-th monomial  of $\omega_1$. Thus, we have the desired equality.
\end{proof}

\begin{proposition} \label{pr:m_n3-1}
        If $\langle \lambda \rangle = ((p+q,p), \emptyset, (1), \emptyset)$ or $\langle \lambda \rangle = ((p+q,p),\emptyset, \emptyset, (1))$, where $p$, $q\geq 0$, then $m_{\langle \lambda \rangle} = (q+1)$ in (\ref{eq:cocharacter_A_gr_*}).
\end{proposition}

\begin{proof}
    We  deal with the case $\langle \lambda \rangle = ((p+q,p), \emptyset, (1), \emptyset)$. The case $\langle \lambda \rangle = ((p+q,p),\emptyset, \emptyset, (1))$ is analogous. 

    Consider Young diagrams of shape $\langle \lambda \rangle$ filled in the standard way.     
From the identity $[y_1^+,y_2^+]$, fixed  the tableaux $ T_{\lambda(3)} = 
    \begin{tabular}{|c|}
        \hline $ t_1 $ \\
        \hline
    \end{tabular}$ for an integer $t_1$, then $t_1$ must be larger than all the
    integers lying in the first $p$ positions of the first row of $T_{\lambda(1)}$ and less than the ones lying in the second row. Otherwise, the corresponding highest weight vector will be a polynomial identity of $A$.

    Thus, the possibilities for the standard Young tableaux, such that  the corresponding   highest weight vectors are linearly independent, are given by 
    $$T_{\lambda(1)} = 
    \begin{tabular}{|c|c|c|c|c|c|c|c|c|}
        \hline$1$ & $2$ & $\cdots$ & $p$ & $\cdots$ & $t_1 -1$ & $t_1 + 1$ & $\cdots$ & $t_2$ \\
        \hline$t_2+1$ & $t_2+2$ & $\cdots$ & $n_1$ & \multicolumn{5}{|c}{}  \\
        \cline { 1 - 4 } 
    \end{tabular} \, ,$$
    $$T_{\lambda(3)} = 
    \begin{tabular}{|c|}
        \hline $ t_1 $ \\
        \hline
    \end{tabular}\,, \qquad T_{\lambda(2)}  = T_{\lambda(4)}  = \emptyset ,$$
    and the corresponding   highest weight vectors are
        \[  \omega_{t_1} = \overline{y}_{1}^{+} \cdots \widetilde{y}_{1}^{+} (y_{1}^{+})^{t_1-1-p} z^{+}  (y_{1}^{+})^{t_2-t_1} \overline{y}_{2}^{+} \cdots \widetilde{y}_{2}^{+}, \]
    where $t_1 = p+1$, \dots, $p + q +1$. 
    We can rewrite the polynomials $\omega_{t_1}$ as 
    \[  \omega_{t_1} = \overline{y}_{1}^{+} \cdots \widetilde{y}_{1}^{+} (y_{1}^{+})^{t_1-p} z^{+}  (y_{1}^{+})^{t_2-p} \overline{y}_{2}^{+} \cdots \widetilde{y}_{2}^{+}, \]
    where $t_1 = p, \dots, p + q $ and $t_1 + t_2 = n_1$. By Equation (\ref{eq:equation}) 
    \[ \omega_{t_1} =  \sum_{j=0}^{p} (-1)^{j} \binom{p}{j} (y_{1}^{+})^{t_1-j} (y_{2}^{+})^{j} z^{+} (y_{1}^{+})^{t_2-p+j} (y_{2}^{+})^{p-j}.\]
    Let us see that the elements of the set $\{ \omega_{t_1} \mid t_1 = p, \dots, p+q \}$ are linearly independent.  Consider the evaluations $z^{+} = e_{1,2} + e_{2,3}$, $y_{i}^{+} = \alpha_i(e_{1,1} + e_{3,3}) + \beta_i e_{2,2}$. Then, \[ 
    \begin{split}
        \omega_{t_1} &  = \bigg[ \sum_{j=0}^{p} (-1)^{j} \binom{p}{j} \alpha_1^{t_1-1}\alpha_2^{j}\beta_1^{t_2-p+j}\beta_2^{p-j}\bigg] e_{1,2} + [*]e_{2,3} \\
        & = \bigg[  \sum_{j=0}^{p} (-1)^{j} \binom{p}{j} \beta_1^{j}\alpha_2^{j}\beta_2^{p-j}\alpha_1^{p-j}\alpha_1^{t_1-p}\beta_1^{t_2 -p} \bigg] e_{1,2} + [*]e_{2,3} \\
        & = \big[ (\beta_2 \alpha_1 - \beta_1 \alpha_2)^{p} \alpha_1^{t_1-p}\beta_1^{t_2 -p} \big] e_{1,2} + [*]e_{2,3} .
    \end{split}  
    \]
Here the coefficients of $e_{23}$ are irrelevant and we denote them by (*).    Now, $(\beta_2 \alpha_1 - \beta_1 \alpha_2)^{p}$ appears in each evaluation of $\omega_{t_1}$, so we consider just the part $\alpha_1^{t_1-p}\beta_1^{t_2 -p}e_{1,2}$. Suppose that there are $\gamma_{i}$'s such that $ \sum_{t_1 = p}^{p+q} \gamma_{t_1}\omega_{t_1} \equiv 0$. Then, for all $\alpha_1$, $\beta_1 \in F$ we have $\sum_{t_1 = p}^{p+q} \gamma_{t_1} \alpha_1^{t_1-p}\beta_1^{t_2 -p} = 0$. It follows  $\gamma_i = 0$ for each $i$. Thus, the highest weight vectors $\omega_{t_1}$ are linearly independent, $t_1 =p$, \dots, $p+q$, and $m_{\langle \lambda \rangle} = (q+1)$.
\end{proof}

Now, we consider the case $n_3 + n_4 = 2$. 

\begin{proposition} \label{pr:m_n3-2}
        If $\langle \lambda \rangle = ((p+q,p), \emptyset, (n_3), (n_4))$, where $p$, $q\geq 0$ and $n_3 + n_4 = 2$, then $m_{\langle \lambda \rangle} = (q+1)$ in (\ref{eq:cocharacter_A_gr_*}).
\end{proposition}

\begin{proof}
    We prove the case $\langle \lambda \rangle = ((p+q,p), \emptyset, (2), \emptyset)$. The cases $\langle \lambda \rangle = ((p+q,p),\emptyset, (1), (1))$ and $\langle \lambda \rangle = ((p+q,p),\emptyset, \emptyset, (2))$ are analogous. 
    Consider Young diagrams of shape $\langle \lambda \rangle$ filled in the standard way.

    As in the proof of Proposition \ref{pr:m_n3-1}, fixed an integer $t_1$ in the tableau $$ T_{\lambda(3)} = 
    \begin{tabular}{|c|c|}
        \hline $ t_1 $ & $ $ \\
        \hline
    \end{tabular}\,,$$  $t_1$ must be larger than  the integers lying in the first $p$ position of the first row of $T_{\lambda(1)}$ and less than all the ones lying in the second row. Also, by the identities $[y^{+},z_{1}z_{2}]$, $[y_{1}^{+},z_{1}y_{2}^{+}z_{2}]$, and  $ z_{1}^{+}y^{+}z_{2}^{+} - z_{2}^{+}y^{+}z_{1}^{+}$, we can fix the integer $n_1+2$ in the second block of the tableaux $T_{\lambda(3)}$.

    Thus, the possibilities for the standard Young tableaux, such that  the corresponding   highest weight vectors are linearly independent, are given by 
    \[
    T_{\lambda(1)} = 
    \begin{tabular}{|c|c|c|c|c|c|c|c|c|}
        \hline$1$ & $2$ & $\cdots$ & $p$ & $\cdots$ & $t_1 -1$ & $t_1 + 1$ & $\cdots$ & $t_2$ \\
        \hline$t_2+1$ & $t_2+2$ & $\cdots$ & $n_1+1$ & \multicolumn{5}{|c}{}  \\
        \cline { 1 - 4 } 
    \end{tabular} \, ,
    \]
    \[
    T_{\lambda(3)} = 
    \begin{tabular}{|c|c|}
        \hline $ t_1 $ & $n_1+2$ \\
        \hline
    \end{tabular}\,, \qquad T_{\lambda(2)}  = T_{\lambda(4)}  = \emptyset ,
    \]
    and the corresponding   highest weight vectors are
    \[  \omega_{t_1} = \overline{y}_{1}^{+} \cdots \widetilde{y}_{1}^{+} (y_{1}^{+})^{t_1-p} z_{1}^{+}  (y_{1}^{+})^{t_2-p} \overline{y}_{2}^{+} \cdots \widetilde{y}_{2}^{+}z_{1}^{+}, \]
        where $t_1 = p$, \dots, $p + q $ and $t_1 + t_2 = n_1$. By Equation (\ref{eq:equation}) 
    \[ \omega_{t_1} =  \sum_{j=0}^{p} (-1)^{j} \binom{p}{j} (y_{1}^{+})^{t_1-j} (y_{2}^{+})^{j} z_{1}^{+} (y_{1}^{+})^{t_2-p+j} (y_{2}^{+})^{p-j}z_{1}^{+}.\]
We prove the elements of the set $\{ \omega_{t_1} \mid t_1 = p, \dots, p+q \}$ are linearly independent.  Consider the evaluations $z_{1}^{+} = e_{1,2} + e_{2,3}$, $y_{i}^{+} = \alpha_i(e_{1,1} + e_{3,3}) + \beta_i e_{2,2}$. Then, \[ 
    \begin{split}
        \omega_{t_1} &  = \bigg[ \sum_{j=0}^{p} (-1)^{j} \binom{p}{j} \alpha_1^{t_1-1}\alpha_2^{j}\beta_1^{t_2-p+j}\beta_2^{p-j}\bigg] e_{1,3} \\
        & = \bigg[  \sum_{j=0}^{p} (-1)^{j} \binom{p}{j} \beta_1^{j}\alpha_2^{j}\beta_2^{p-j}\alpha_1^{p-j}\alpha_1^{t_1-p}\beta_1^{t_2 -p} \bigg] e_{1,3} \\
        & = \big[ (\beta_2 \alpha_1 - \beta_1 \alpha_2)^{p} \alpha_1^{t_1-p}\beta_1^{t_2 -p} \big] e_{1,3} .
    \end{split}  
    \]
    Since $(\beta_2 \alpha_1 - \beta_1 \alpha_2)^{p}$ appears in each evaluation of $\omega_{t_1}$,  we consider only the part $\alpha_1^{t_1-p}\beta_1^{t_2 -p}e_{1,2}$. Suppose  there exist $\gamma_{i}$'s such that $\sum_{t_1 = p}^{p+q} \gamma_{t_1}\omega_{t_1} \equiv 0$. Then, for all $\alpha_1$, $\beta_1 \in F$ one has $\sum_{t_1 = p}^{p+q} \gamma_{t_1} \alpha_1^{t_1-p}\beta_1^{t_2 -p} = 0$.  It follows that $\gamma_i = 0$ for each $i$. So the highest weight vectors $\omega_{t_1}$ are linearly independent for $t_1 =p$, \dots, $p+q$ and $m_{\langle \lambda \rangle} = (q+1)$.
\end{proof}

Hence we obtain the following theorem.

\begin{theorem}
Let 
\[ \chi_{n_1,n_2,n_3,n_4}(A) = \sum_{\langle \lambda \rangle \vdash (n_1,n_2,n_3,n_4) } m_{\langle \lambda \rangle}\chi_{\lambda(1)} \otimes  \chi_{\lambda(2)} \otimes  \chi_{\lambda(3)} \otimes  \chi_{\lambda(4)} \]
be the $(n_1,n_2,n_3,n_4)$-cocharacter of $A$. Then
\begin{itemize}
    \item[(1)] $m_{\langle \lambda \rangle} = 1$, if $\langle \lambda \rangle = ((n_1), \emptyset, \emptyset, \emptyset) $ or $\langle \lambda \rangle  = (\emptyset, \emptyset, (n_3),(n_4))$, where $n_1 > 0$ and $0< n_3 + n_4 \leq 2$.

    \item[(2)] $m_{\langle \lambda \rangle} = q+1$, if $\lambda = ((p+q,p), \emptyset, (n_3), (n_4))$,   where  $p,q \geq 0$ and $n_3 + n_4 = 1$;

    \item[(3)] $m_{\langle \lambda \rangle} = q+1$, if $\lambda = ((p+q,p), \emptyset, (n_3), (n_4))$, where $p,q \geq 0$ and $n_3 + n_4 = 2$.    
\end{itemize}
    In all remaining cases $m_{\langle \lambda \rangle} = 0$.
\end{theorem}

\bibliographystyle{amsplain}
\bibliography{references}

\end{document}